\documentclass[preprint,12pt,sort]{elsarticle}

\usepackage{epstopdf}
\usepackage{float}
\usepackage{subfigure}

\usepackage{amssymb}
\usepackage{amsthm}
\usepackage{amsmath}
\usepackage{amsfonts}
\usepackage{mathrsfs}
\usepackage[colorlinks, citecolor=blue]{hyperref}

\numberwithin{equation}{section}

\newtheorem{thm}{Theorem}[section]

\newtheorem{prop}[thm]{Proposition}

\newdefinition{defn}{Definition}[section]
\newdefinition{conj}{Conjecture}[section]
\newdefinition{exmp}{Example}[section]

\newtheorem{rem}{Remark}[section]

\journal{Journal of Computational Physics}

\begin{document}
\begin{frontmatter}



\title{A Bloch decomposition-based stochastic Galerkin method for quantum dynamics with a random external potential\tnoteref{t1}}
\tnotetext[t1]{This work was partially supported by the NSFC Projects No.~11322113, ~91330203.}


\author[rvt]{Zhizhang Wu}
\ead{wzz14@mails.tsinghua.edu.cn}

\author[rvt]{Zhongyi Huang \corref{cor1}}
\ead{zhuang@math.tsinghua.edu.cn}

\cortext[cor1]{Corresponding author.}
\address[rvt]{Department of Mathematical Sciences, Tsinghua University, Beijing, 100084, China}

\begin{abstract}
In this paper, we consider the numerical solution of the one-dimensional Schr\"{o}dinger equation with a periodic lattice potential and a random external potential. This is an important model in solid state physics where the randomness is involved to describe some complicated phenomena that are not exactly known. Here we generalize the Bloch decomposition-based time-splitting pseudospectral method to the stochastic setting using the generalize polynomial chaos with a Galerkin procedure so that the main effects of dispersion and periodic potential are still computed together. We prove that our method is unconditionally stable and numerical examples show that it has other nice properties and is more efficient than the traditional method. Finally, we give some numerical evidence for the well-known phenomenon of Anderson localization.
\end{abstract}


\end{frontmatter}


\section{Introduction}
\label{sec 1}
An important problem in solid state physics is to describe the motion of electrons within periodic potentials generated by the ionic cores. With the emergence of this novel structure, this problem has been studied from a physical as well as a mathematical point of view and extensive results have been achieved \cite{Asch1998,Blount1962,Luttinger1951,Panati2003,Hoverman2001,Zak1968}. One of the most brilliant ideas is to combine the dispersion and periodic lattice effects together, which results in a new energy band structure, known as the Bloch band \cite{Bloch1928}. The motivation of such a combination is a separation of scales in this problem, where the external fields vary on much larger scales than the periodic potentials and can be considered weak in the comparison to the periodic fields \cite{Ashcroft1976}.

To mathematically formulate this problem, consider the one-dimensional deterministic Schr\"{o}dinger equation for electrons in a semiclassical asymptotic scaling,
\begin{equation} \label{deterministic Schrodinger}
\left\{
\begin{aligned}
i\varepsilon\partial_t\psi & = -\frac{\varepsilon^2}{2}\partial_{xx}\psi+V_{\Gamma}(\frac{x}{\varepsilon})\psi+U(x)\psi, &  & x\in\mathbb{R},  t\in\mathbb{R}, \\
\psi|_{t=0} & = \psi_{in}(x).
\end{aligned}
\right.
\end{equation}
where $0< \varepsilon \ll1$ is the small semiclassical parameter describing the microscopic/macroscopic scale ratio. $U(x)\in\mathbb{R}$ is the external potential and the highly oscillating lattice potential $V_{\Gamma}(y)\in\mathbb{R}$ is assumed to be periodic with respect to some regular lattice $\Gamma$. For definiteness we may assume that
\begin{equation} \label{periodic potential}
V_{\Gamma}(y+2\pi)=V_{\Gamma}(y), \quad \forall y \in \mathbb{R},
\end{equation}
i.e., $\Gamma=2\pi\mathbb{Z}$.

There have also been some researches on the numerical methods for this problem \cite{Gosse I,Gosse II,Gosse III,Jin2011}. Among these methods, the Bloch decomposition-based time-splitting pseudospectral method (BD), developed by Huang et al \cite{Huang2007,Huang2008,Huang2010}, is based on the classical Bloch decomposition theory. A prominent advantage of this method is that it converges with $\Delta x=O(\varepsilon)$ and $\Delta t=O(1)$ and hence it works better in the case where $\varepsilon \ll1$ than other traditional methods. Furthermore, it comprises spectral convergence for the spacial discretization and second order convergence in time.

If the medium in which electrons move is disordered like amorphous solids and random alloys \cite{Leschke2005,Johnson1986,Kirsh1989} or the lattices are influenced by thermal fluctuation or randomly distributed impurities \cite{Fisher1991,Fisher1991Thermal,Avishai1992,Halperin1965}, the equation we consider should convert from \eqref{deterministic Schrodinger} to a different but similar one with random potential \cite{Anderson1958,Roati2008}, which will be mathematically formulated below. This is also a basic problem in quantum mechanics. The random Schr\"{o}dinger operators have been intensively studied in a theoretical way \cite{Kirsh1982,Kirsh1982Spectrum,Bal1999,Bal2011,Fouqe2007}. However, the numerical literatures on them, especially on the problem \eqref{deterministic Schrodinger} with random potentials are not so abundant \cite{Bal2006}. Indeed, this kind of problems belongs to uncertainty quantification (UQ) for PDEs with randomness, where randomness could appear in initial conditions, boundary conditions, coefficients of the equations, etc. There have been quite a few numerical methods developed for UQ in recent years. The Monte Carlo (MC) method and the stochastic collocation (SC) method \cite{Xiu2005,Xiu2009,Xiu2010} are two of the most popular methods based on sampling and repetitive calls of deterministic solvers, which draw statistical information from the ensemble of solutions. Another popular method is the generalized polynomial chaos (gPC) approach \cite{Xiu2002}, which is non-sampling and is a generalization of the Wiener-Hermite polynomial chaos expansion. Combined with the stochastic Galerkin method, it has been successfully applied to many physical and engineering problems \cite{Hu2015,Jin2015}, where fast convergence can be observed if the solution is sufficiently smooth.

In this paper, we mainly focus on the problem of motions of electrons in a periodic potential influenced by some weakly random factors where the periodic lattice structure is not changed, (e.g. a random electric field, etc), so that the randomness can be restricted to the external potential $U(x)$ in \eqref{deterministic Schrodinger}. To cope with this problem, we combine the Bloch decomposition-based time-splitting method with the generalized polynomial chaos expansion. With the mathematical setting and stochastic Galerkin procedure, we would obtain a deterministic PDE system, to which we can easily generalize the BD algorithm. This new method, named as the Bloch decomposition-based stochastic Galerkin method (BD-SG), preserves the advantage of allowing a relatively larger time step size and comprising spectral convergence for the spatial discretization, second order convergence in time and fast convergence in gPC order. We shall also prove that this method enjoys the property of weak conservation of mass and is unconditionally stable. Our numerical examples will show that our method is efficient and even enjoys the property of weak conservation of energy. Finally, as Anderson localization \cite{Anderson1958,Frohlich1983,Roati2008} is an important phenomenon in wave propagation in disordered media, we give some numerical evidence to this phenomenon.

The paper is organized as follows: In Section \ref{sec 2}, we briefly review the Bloch decomposition-based time-splitting method. In Section \ref{sec 3}, we present our new BD-SG scheme and introduce the classical time-splitting stochastic collocation method as a comparison. In Section \ref{sec 4}, we analyze the properties of our algorithm. In Section \ref{sec 5}, numerical examples are presented to show the feasibility and nice properties of our algorithm. Also, numerical evidence of Anderson localization is included. Finally, we give a conclusion in Section \ref{sec 6}.

\section{Review of Bloch decomposition-based method}
\label{sec 2}
In this section, we will briefly recapitulate the numerical method developed in \cite{Huang2007}. For the convenience of the reader we first recall some basic definitions and important facts to be used in dealing with the periodic Schr\"{o}dinger operator.

\subsection{Review of Bloch decomposition}
For the sake of simplicity, set $y=x/\varepsilon$. With $V_{\Gamma}$ obeying \eqref{periodic potential}, we have the following \cite{Ashcroft1976}:
\begin{itemize}
\item the fundamental domain of our lattice $\Gamma=2\pi\mathbb{Z}$ is $\mathcal{C}=(0,2\pi)$.
\item the dual lattice $\Gamma^*$ is then simply given by $\Gamma^*=\mathbb{Z}$.
\item the fundamental domain of the dual lattice, i.e., the (first) Brillouin zone, is $\mathcal{B}=(-\frac{1}{2},\frac{1}{2})$.
\end{itemize}

Next, let us consider the eigenvalue problem
\begin{equation} \label{Bloch eigen problem}
\left\{
\begin{aligned}
\left(-\frac{1}{2}\partial_{yy}+V_{\Gamma}(y)\right)\varphi_m(y,k)& =  E_m(k)\varphi_m(y,k), \\
\varphi_m(y+2\pi,k) & =  e^{i2\pi ky}\varphi_m(y,k), \quad \forall k \in \mathcal{B}.
\end{aligned}
\right.
\end{equation}
It's well known \cite{Teufel2003,Wilcox1978} that under very mild conditions on $V_{\Gamma}$, the problem \eqref{Bloch eigen problem} has a complete set of eigenfunctions $\varphi_m(y,k)$, $m \in \mathbb{N}$, which is an orthonormal basis in $L^2(\mathcal{C})$ for each fixed $k\in\bar{\mathcal{B}}$. Correspondingly there exists a countable family of real eigenvalues that can be ordered according to $E_1(k)\leq E_2(k)\leq\cdots\leq E_m(k)\leq\cdots,\quad m\in\mathbb{N}$. The set $\{E_m(k)|k\in\mathcal{B}\}\subset\mathbb{R}$ is called the $m$th energy band of the operator $H$ \cite{Bloch1928}.

For convenience we will usually rewrite $\varphi_m(y,k)$ as
\begin{equation}
\varphi_m(y,k)=e^{iky}\chi_m(y,k), \quad \forall m\in\mathbb{N},
\end{equation}
where now $\chi_m(\cdot,k)$ is $2\pi$-periodic and called a Bloch function. In terms of $\chi_m(y,k)$, the eigenvalue problem \eqref{Bloch eigen problem} converts to
\begin{equation} \label{Bloch eigen problem 2}
\left\{
\begin{aligned}
H(k)\chi_m(y,k)& = E_m(k)\chi_m(y,k), \\
\chi_m(y+2\pi,k) & = \chi_m(y,k), \quad \forall k \in \mathcal{B}.
\end{aligned}
\right.
\end{equation}
where
\begin{equation}
H(k):=\frac{1}{2}(-i\partial_y+k)^2+V_{\Gamma}(y)
\end{equation}
denotes the so-called shifted Hamiltonian.

By solving the eigenvalue problem \eqref{Bloch eigen problem}, the Bloch decomposition allows us to decompose the Hilbert space $\mathcal{H}=L^2(\mathbb{R})$ into a direct sum of orthogonal band spaces \cite{Markowich1994,Wilcox1978}, i.e.,
\begin{equation}
L^2(\mathbb{R})=\displaystyle\bigoplus_{m=1}^\infty \mathcal{H}_m, \quad \mathcal{H}_m :=\left\{f_m(y)=\int_{\mathcal{B}}g(k)\varphi_m(y,k)dk,g\in L^2(\mathcal{B})\right\}.
\end{equation}
This leads to
\begin{equation}
\forall f\in L^2(\mathbb{R}),\quad f(y)=\sum_{m\in\mathbb{N}}f_m(y),\quad f_m\in\mathcal{H}_m.
\end{equation}
The corresponding projection of $f$ onto the $m$th band space is given by \cite{Markowich1994}
\begin{equation}
f_m(y)\equiv (\mathbb{P}_m f)(y)=\int_{\mathcal{B}}\left(\int_{\mathbb{R}}f(\zeta)\bar{\varphi}_m(\zeta,k)d\zeta\right)\varphi_m(y,k)dk.
\end{equation}
In what follows, we will denote by
\begin{equation} \label{Bloch coefficient}
C_m(k):=\int_{\mathbb{R}}f(\zeta)\bar{\varphi}_m(\zeta,k)d\zeta
\end{equation}
the coefficient of the Bloch decomposition. The Bloch decomposition reduces the equation
\begin{equation}
i\partial_t\psi=-\frac{1}{2}\partial_{yy}\psi+V_{\Gamma}(y)\psi, \quad \psi |_{t=0}=\psi_{in}(y),
\end{equation}
into countably many, exactly solvable problems on $\mathcal{H}_m$. In each band space, we simply obtain
\begin{equation} \label{band ode}
i\partial_t\psi_m=E_m(-i\partial_y)\psi_m, \quad \psi_m |_{t=0}=(\mathbb{P}_m\psi_{in})(y),
\end{equation}
where $E_m(-i\partial_y)$ denotes the Fourier multiplier corresponding to the symbol $E_m(k)$. By using the Fourier transformation $\mathcal{F}$, \eqref{band ode} is solved by
\begin{equation} \label{Fourier evolution of mth band}
\psi_m(t,y)=\mathcal{F}^{-1}\left(e^{-iE_m(k)t}(\mathcal{F}(\mathbb{P}_m\psi_{in}))(k)\right).
\end{equation}
Here the energy band $E_m(k)$ is understood to be periodically extended to all of $\mathbb{R}$.

\subsection{The Bloch decomposition-based time-splitting algorithm}
We shall recall here the most important steps of this algorithm.
As a necessary preprocessing step, we first need to calculate the energy band $E_m(k)$ and eigenfunctions $\varphi_m(y,k)$ from \eqref{Bloch eigen problem} or equivalently \eqref{Bloch eigen problem 2}. To numerically solve the eigenvalue problem \eqref{Bloch eigen problem 2}, we will approximate it by an algebraic eigenvalue problem (for more details, please refer to \cite{Huang2007}). Upon solving this algebraic eigenvalue problem, we will obtain $\{\hat{\chi}_m(\lambda,k)|\lambda \in \{-\Lambda,\ldots,\Lambda-1\}\subset\mathbb{Z}\}$, the Fourier coefficient of $\chi_m(y,k)$, with which we can reconstruct $\chi_m(y,k)$ for each $k$.

For convenience of the computations, we consider \eqref{deterministic Schrodinger} on a bounded domain $\mathcal{D}=[0,2\pi]$ with periodic boundary conditions, which approximates the one-dimensional whole-space problem as long as the observed wave function does not touch the boundaries $x=0,2\pi$. Then for some $N\in\mathbb{N},t>0$, let the time step be $\Delta t=t/N$ and $t_n=n\Delta t,\quad n=1,\ldots,N$. Suppose further that there are $L\in\mathbb{N}$ lattice cells of $\Gamma$ within $\mathcal{D}$ and $R\in\mathbb{N}$ grid points in each lattice cell, which yields the following discretization
\begin{equation}
\left\{
\begin{aligned}
k_\ell& = & -\frac{1}{2}+\frac{\ell-1}{L}, \quad where \quad \ell\in\{1,\ldots,L\}\subset\mathbb{N},\\
y_r & = & \frac{2\pi(r-1)}{R}, \quad where \quad r\in\{1,\ldots,R\}\subset\mathbb{N}.
\end{aligned}
\right.
\end{equation}
Thus we evaluate the solution at the grid points
\begin{equation}
x_{\ell,r}=\varepsilon\big(2\pi(\ell-1)+y_r\big).
\end{equation}
Now we introduce the following unitary transformation of $f\in L^2(\mathbb{R})$
\begin{equation} \label{unitary transformation}
(Tf)(y,k)\equiv\tilde{f}(y,k):=\sum_{\gamma\in\mathbb{Z}}f\big(\varepsilon(y+2\pi\gamma)\big)e^{-i2\pi k\gamma}, \quad y\in\mathcal{C},k\in\mathcal{B},
\end{equation}
such that $\tilde{f}(y,k)$ has the same periodicity properties w.r.t. $y$ and $k$ as the Bloch eigenfunctions $\varphi_m(y,k)$ and can be decomposed into a linear combination of such eigenfunctions. Moreover, it can be shown that the Bloch coefficient in \eqref{Bloch coefficient} can be equivalently written as
\begin{equation} \label{comptation of Bloch coefficient}
C_m(k)=\int_{\mathcal{C}}\tilde{f}(y,k)\bar{\varphi}_m(y,k)dy.
\end{equation}

We are now ready to set up the Bloch decomposition-based time-splitting algorithm. Suppose that at time $t_n$ we are given $\psi(t_n,x_{\ell,r})\approx\psi_{\ell,r}^n$. Then $\psi_{\ell,r}^{n+1}$ is obtained as follows:

\emph{Step 1.} Solve the equation
\begin{equation}
i\varepsilon\partial_t\psi=-\frac{\varepsilon^2}{2}\partial_{xx}\psi+V_{\Gamma}\left(\frac{x}{\varepsilon}\right)\psi
\end{equation}
on a fixed time interval $\Delta t$. Here we need to apply the transformation $T$ defined in \eqref{unitary transformation} to $\psi$ for each fixed $t$ so that $T\psi(t,\cdot)\equiv\tilde{\psi}(t,y,k)$ satisfies the same periodic boundary conditions w.r.t. $y$ as $\varphi_m(y,k)$, where $y=x/\varepsilon$. Then $\tilde{\psi}(t,y,k)$ can be decomposed as
\begin{equation}
\tilde{\psi}(t,y,k)=\sum_{m\in\mathbb{N}}\mathbb{P}_m\tilde{\psi}=\sum_{m\in\mathbb{N}}C_m(t,k)\varphi_m(y,k).
\end{equation}
By \eqref{band ode}, we have the following evolution equation for the Bloch coefficient $C_m(t,k)$
\begin{equation}
i\varepsilon\partial_t C_m(t,k)=E_m(k)C_m(t,k)
\end{equation}
which yields
\begin{equation} \label{Bloch coefficient evolution}
C_m(t,k)=C_m(0,k)e^{-iE_m(k)t/\varepsilon}.
\end{equation}

\emph{Step 2.} Solve the ODE
\begin{equation}
i\varepsilon\partial_t\psi=U(x)\psi
\end{equation}
on the same time interval where the solution in Step 1 serves as the initial condition here. And the exact solution is
\begin{equation}
\psi(t,x)=\psi(0,x)e^{-iU(x)t/\varepsilon}.
\end{equation}
The algorithm given above is first order in time but we can easily obtain a second order scheme by the Strang's splitting method. Indeed, Step 1 consists of several intermediate steps as given below:

\emph{Step 1.1.} Compute $\tilde{\psi}$ at $t_n$ by
\begin{equation} \label{step 1.1}
\tilde{\psi}_{\ell,r}^n=\sum_{j=1}^L \psi_{j,r}^n e^{-i2\pi k_{\ell}(j-1)}.
\end{equation}

\emph{Step 1.2.} Compute the Bloch coefficients via \eqref{comptation of Bloch coefficient},
\begin{equation}
\begin{split}
C_m(t_n,k_{\ell})\approx C_{m,\ell}^n & =\frac{2\pi}{R}\sum_{r=1}^R \tilde{\psi}_{m,\ell}^n \overline{\varphi_m}(y_r,k_\ell) \\
 & =\frac{2\pi}{R}\sum_{r=1}^R \tilde{\psi}_{m,\ell}^n \overline{\chi_m}(y_r,k_\ell)e^{-ik_\ell y_r} \\
 & \approx \frac{2\pi}{R}\sum_{r=1}^R \tilde{\psi}_{m,\ell}^n \sum_{\lambda=-R/2}^{R/2} \overline{\hat{\chi}_m}(\lambda,k_\ell)e^{-i(k_\ell+\lambda)y_r}.
\end{split}
\end{equation}

\emph{Step 1.3.} Compute the Bloch coefficients at $t_{n+1}$ via \eqref{Bloch coefficient evolution}
\begin{equation}
C_{m,\ell}^{n+1}=C_{m,\ell}^n e^{-iE_m(k_\ell)\Delta t/\varepsilon}.
\end{equation}

\emph{Step 1.4.} Obtain $\tilde{\psi}$ at $t^{n+1}$ by summing up all the band contributions
\begin{equation}
\tilde{\psi}_{\ell,r}^{n+1}=\sum_{m=1}^M C_{m,\ell}^{n+1} \chi_m(y_r,k_\ell) e^{ik_\ell y_r}.
\end{equation}

\emph{Step 1.5.} Implement the inverse transformation
\begin{equation} \label{step 1.5}
\psi_{\ell,r}^{n+1}=\frac{1}{L}\sum_{j=1}^L \tilde{\psi}_{j,r}^{n+1} e^{i2\pi k_j(\ell-1)}.
\end{equation}
This concludes the numerical procedure in Step 1.

\subsection{Stability of the algorithm}
It has been mentioned in some references \cite{Huang2007,Huang2008} that this splitting scheme conserves the total mass $||\psi(t,x)||_{L^2}$ on the fully discrete level and hence it's unconditionally stable. But no rigorous proofs have been found. So we are trying to make up for this.

It's easy to see that Step 2 conserves the total mass since $U(x)\in\mathbb{R}$. As for Step 1, by \eqref{Fourier evolution of mth band} and the orthogonality of the Bloch eigenfunctions $\varphi_m$, it's natural that Step 1 should conserve the total mass. To be precise, we give the following proposition.
\begin{prop} \label{prop of 1d conservation}
Assume that the number of bands is the same as that of grid points in each lattice, i.e., $M=R$, then Step 1 in the Bloch decomposition based time-splitting scheme conserves the total mass, i.e., if $\psi^{n+1}$ is obtained from $\psi^n$ through Step 1, then $||\psi^{n+1}||_{L^2}=||\psi^n||_{L^2}$.
\end{prop}
\begin{proof}
We have
\begin{equation}
\begin{split}
\frac{LR}{2\pi}||\psi^{n+1}||_{L^2}^2 & =\sum_{\ell=1}^L\sum_{r=1}^R |\psi_{\ell,r}^{n+1}|^2=\sum_{\ell=1}^L\sum_{r=1}^R \left|\frac{1}{L}\sum_{j=1}^L \tilde{\psi}_{j,r}^{n+1} e^{i2\pi k_j(\ell-1)}\right|^2 \\
 & = \frac{1}{L}\sum_{j=1}^L\sum_{r=1}^R |\tilde{\psi}_{j,r}^{n+1}|^2=\frac{1}{L}\sum_{j=1}^L\sum_{r=1}^R \left|\sum_{m=1}^M C_{m,j}^{n+1}\varphi_m(y_r,k_j)\right|^2 \\
 & = \frac{R}{2\pi}\frac{1}{L}\sum_{j=1}^L\sum_{m=1}^M |C_{m,j}^{n+1}|^2=\frac{R}{2\pi}\frac{1}{L}\sum_{j=1}^L\sum_{m=1}^M |C_{m,j}^n e^{-iE_m(k_j)\Delta t/\varepsilon}|^2\\
 & = \frac{R}{2\pi}\frac{1}{L}\sum_{j=1}^L\sum_{m=1}^M |C_{m,j}^n|^2 \\
 & =\frac{R}{2\pi}\frac{1}{L}\sum_{j=1}^L\sum_{m=1}^M \left|\frac{2\pi}{R}\sum_{r=1}^R \tilde{\psi}_{m,j}^n \overline{\varphi_m}(y_r,k_j)\right|^2 \\
 & = \frac{1}{L}\sum_{j=1}^L\sum_{r=1}^R |\tilde{\psi}_{j,r}^n|^2=\frac{1}{L}\sum_{j=1}^L\sum_{r=1}^R \left|\sum_{\ell=1}^L \psi_{\ell,r}^n e^{-i2\pi k_j(\ell-1)}\right|^2 \\
 & = \sum_{\ell=1}^L\sum_{r=1}^R |\psi_{\ell,r}^n|^2=\frac{LR}{2\pi}||\psi^n||_{L^2}^2.
\end{split}
\end{equation}
Here we use the identities
\begin{equation}
\sum_{\ell=1}^L e^{i2\pi(j-k)\ell/L}=\left\{
\begin{aligned}
L, & & (j-k)\mod L=0,\\
0, & & (j-k)\mod L\neq0,
\end{aligned}
\right.
\end{equation}
and
\begin{equation} \label{discrete Bloch eigenfunction orthogonality}
\sum_{r=1}^R \varphi_m(y_r,k_j)\overline{\varphi_n(y_r,k_j)}=\frac{R}{2\pi}\delta_{mn},
\end{equation}
\begin{equation} \label{orthogonality wrt band}
\sum_{m=1}^M \varphi_m(y_r,k_j)\overline{\varphi_m(y_s,k_j)}=\frac{R}{2\pi}\delta_{rs}.
\end{equation}
\eqref{discrete Bloch eigenfunction orthogonality} is the orthogonality of $\varphi_m(y,k)$ w.r.t. $y$ on the discrete level and \eqref{orthogonality wrt band} is due to
\begin{equation} \label{ifft to reconstruct phi}
\varphi_m(y_r,k_j)=\frac{1}{\sqrt{2\pi}}\sum_{\lambda=-R/2}^{R/2-1}\hat{\chi}_m(\lambda,k_j)e^{i(k_j+\lambda)y_r}
\end{equation}
and
\begin{equation}
\sum_{\lambda=-R/2}^{R/2-1} e^{i2\pi(j-k)\lambda/R}=\left\{
\begin{aligned}
R, & & (j-k)\mod R=0\\
0, & & (j-k)\mod R\neq0
\end{aligned}
\right.
\end{equation}
where the multiplier $\frac{1}{\sqrt{2\pi}}$ in \eqref{ifft to reconstruct phi} is the scaling in the preprocessing step to make \eqref{discrete Bloch eigenfunction orthogonality} hold and if $M=R$, $[\hat{\chi}_m(\lambda,k_j)]$ is, for fixed $j$, with $m$ and $\lambda$ being the indices, the matrix of eigenvectors of a Hermite matrix which is unitary since the shifted Hamiltonian $H(k_j)$ is self-adjoint (for more details, please refer to \cite{Huang2007}).
\end{proof}

\section{BD-based stochastic Galerkin method}
\label{sec 3}
In this section, we are about to generalize the Bloch decomposition-based time-splitting method to the Schr\"{o}dinger equation subject to random inputs. We first set up the mathematical formulation of the problem we are considering.

\subsection{Mathematical formulation}
Recall that we focus on the case where electrons move in a periodic potential influenced by a weakly random external field which does not change the periodic lattice structure. In this case, the randomness will be restricted to the external potential. Now the external potential should depend on not only the spatial variable, but also a $d$-dimensional random variable, i.e., $U\equiv U(x,z)$, where $z\in\mathbb{R}^d$ represents some kind of randomness. As a result, the solution is also dependent on the random variable $z$, i.e., $\psi\equiv\psi(t,x,z)$.

So we are now considering the following problem,
\begin{equation} \label{Schrodinger}
\left\{
\begin{aligned}
i\varepsilon\partial_t\psi & = -\frac{\varepsilon^2}{2}\partial_{xx}\psi+V_{\Gamma}(\frac{x}{\varepsilon})\psi+U(x,z)\psi, &  & x\in\mathbb{R},  t\in\mathbb{R}, z\in\Omega\subset\mathbb{R}^d, \\
\psi|_{t=0} & = \psi_{in}(x)
\end{aligned}
\right.
\end{equation}
where it's assumed that $\psi_{in}\in L^2(\mathbb{R})$ and $||\psi_{in}||_{L^2}=1$ without loss of generosity.

\subsection{gPC Galerkin method}
In view of the generalized polynomial chaos expansion, we approximate the external potential and the solution via a truncated orthogonal polynomial series \cite{Xiu2002}. That is, for random variable $z\in\mathbb{R}^d$, we approximate them by
\begin{equation} \label{gPC1}
U(x,z)\approx U_Q(x,z)=\sum^P_{p=1}\hat{U}_p(x)\Phi_p(z),
\end{equation}
\begin{equation} \label{gPC2}
\psi(t,x,z)\approx \psi_Q(t,x,z)=\sum^P_{p=1}\hat{\psi}_p(t,x)\Phi_p(z)
\end{equation}
where $\{\Phi_p(z)\}$ are from $\mathbb{P}^d_Q$, the space of $d$-variate orthogonal polynomials of degree up to $Q\geq1$, and orthonormal:
\begin{equation}
\int\Phi_k(z)\Phi_j(z)d\mu(z)=\delta_{kj},\quad 1\leq k,j\leq P=dim(\mathbb{P}^d_Q)=C^d_{d+Q}.
\end{equation}
Here $\mu(z)$ is the probability distribution of $z$ and $\delta_{kj}$ is the Kronecker delta function. The orthogonality with respect to $\mu(z)$ defines the orthogonal polynomials with respect to the same weight function to achieve faster convergence. For example, Gaussian distribution defines Hermite polynomials; Gamma distribution defines Laguerre polynomials, etc. Note that when the dimension of randomness $d>1$, an ordering scheme for multiple index is needed to re-order the polynomials into a single index $p$. Typically, the graded lexicographic order is used. See Section 5.2 of \cite{Xiu2010}.

Once a basis is chosen, the gPC approximations \eqref{gPC1} and \eqref{gPC2} are inserted into the one-dimensional random Schr\"{o}dinger equation \eqref{Schrodinger}. Then a Galerkin projection procedure can be applied to ensure the residue is orthogonal to $\mathbb{P}^d_Q$. That is, for $p=1,\ldots,P$,
\begin{equation}
\mathbb{E}[i\varepsilon\partial_t\psi_Q\Phi_p]=\mathbb{E}[-\frac{\varepsilon^2}{2}\partial_{xx}\psi_Q\Phi_p]
+\mathbb{E}[V_{\Gamma}(\frac{x}{\varepsilon})\psi_Q\Phi_p]+\mathbb{E}[U_Q(x,z)\psi_Q\Phi_p]
\end{equation}
where $\mathbb{E}$ is the expectation operator. By orthogonality, we have
\begin{equation}
i\varepsilon\partial_t\hat{\psi}_p=-\frac{\varepsilon^2}{2}\partial_{xx}\hat{\psi}_p
+V_{\Gamma}(\frac{x}{\varepsilon})\hat{\psi}_p+\sum^P_{j=1}\sum^P_{q=1}\hat{U}_j\hat{\psi}_qe_{jqp}
\end{equation}
where $e_{jqp}=\mathbb{E}[\Phi_j\Phi_q\Phi_p]$.

Hence, we obtain the following system
\begin{equation} \label{PDE system}
i\varepsilon\partial_t\vec{\psi}=-\frac{\varepsilon^2}{2}\partial_{xx}\vec{\psi}
+V_{\Gamma}(\frac{x}{\varepsilon})\vec{\psi}+A_U\vec{\psi}, \quad x \in \mathbb{R}, \quad t \in \mathbb{R}
\end{equation}
where $\vec{\psi}=(\hat{\psi}_1,\ldots,\hat{\psi}_P)^T$ is the coefficient vector, and $A_U=(a_{pq})_{1\leq p,q\leq P}$ with
\begin{equation}
a_{pq}(x)=\sum^P_{j=1}\hat{U}_j(x)e_{jqp}.
\end{equation}
Obviously, $A_U$ is real and symmetric.

Since the initial condition for \eqref{Schrodinger} is deterministic, the initial conditions for \eqref{PDE system} are
\begin{equation} \label{initial condition for random Schrodinger}
\hat{\psi}_p(0,x)=\left\{
\begin{aligned}
\psi_{in}(x), & & p=1,\\
0, & & p\neq 1.
\end{aligned}
\right.
\end{equation}

\begin{rem}
It's straightforward to generalize the above procedure to the case where randomness is also included in the initial condition, i.e, $\psi_{in}\equiv\psi_{in}(x,z)$, which may be due to the uncertainty of measurement. Then the initial condition for the PDE system \eqref{Schrodinger} changes from \eqref{initial condition for random Schrodinger} to
\begin{equation} \label{initial condition for random Schrodinger 2}
\hat{\psi}_p(0,x)=\int\psi_{in}(x,z)\Phi_p(z)d\mu(z), \quad p=1,\cdots,P.
\end{equation}
The other settings remain the same.
\end{rem}

\subsection{Bloch decomposition-based stochastic Galerkin scheme}
Using gPC expansion, what we want to solve changes from a scalar stochastic PDE \eqref{Schrodinger} to a deterministic PDE system \eqref{PDE system}. Since the periodic lattice potential $V_{\Gamma}$ is not coupled in the system \eqref{PDE system}, it's easy to generalize the Bloch decomposition-based time-splitting scheme to solve it.

As a preprocessing step, we still need to calculate the energy band $E_m(k)$ and eigenfunctions $\varphi_m(y,k)$. In addition, the matrix $A_U(x)$ is also needed. Note that these computations are needed only once as a preparatory step, the numerical costs for them are negligible.

Suppose that at time $t_n$ we are given $\vec{\psi}(t_n,x_{l,r}) \approx \vec{\psi}^n_{l,r}$. Then $\vec{\psi}^{n+1}_{l,r}$ is obtained as follows:

\emph{Step 1.} First, we solve $P$ equations
\begin{equation} \label{step 1 of BDSG}
i\varepsilon\partial_t\hat{\psi}_p=-\frac{\varepsilon^2}{2}\partial_{xx}\hat{\psi}_p
+V_{\Gamma}(\frac{x}{\varepsilon})\hat{\psi}_p, \quad p=1,\ldots,P
\end{equation}
on a fixed time interval $\Delta t$ using Bloch decomposition method. In other words, we call $P$ times the Step 1 \eqref{step 1.1}-\eqref{step 1.5} in the BD-based time-splitting method, with \eqref{initial condition for random Schrodinger} or \eqref{initial condition for random Schrodinger 2} as the initial condition at $t_0$.

\emph{Step 2.} In the second step, we turn to solve the ODE system
\begin{equation} \label{ode system}
i\varepsilon\partial_t\vec{\psi}=A_U\vec{\psi}
\end{equation}
on the same time interval, where the solution obtained in Step 1 serves as the initial condition for Step 2. We can easily obtain the analytic solution for the system, which formally can be represented as
\begin{equation} \label{ode solution}
\vec{\psi}(t,x)=e^{-iA_U(x)t/\varepsilon}\vec{\psi}(0,x).
\end{equation}

\subsection{A classical time-splitting spectral scheme with stochastic collocation\\ method}
Although there's not much numerical literature on this specific problem, there're some sampling methods for uncertainty quantification if a deterministic solver already exists \cite{Xiu2009,Xiu2010}. Here, we introduce the stochastic collocation method with the classical time-splitting spectral method (TS) as the deterministic solver (TS-SC) as a comparison to the BD-SG scheme.

Stochastic collocation, unlike Monte Carlo method, makes use of the polynomial approximation theory to strategically locate the sample nodes to gain accuracy while sampling. In one-dimensional case, to gain high accuracy, the optimal choice is usually the Gauss quadratures. For dimensions larger than 1, two popular approaches are the tensor products of one-dimensional nodal sets and sparse grids \cite{Xiu2005}.

Let $\{z_j\}_{j=1}^{N_{sc}}$ be the set of nodes, where $N_{sc}$ is the total number of nodes. One may then apply the classical time-splitting spectral scheme to \eqref{Schrodinger} for each fixed $z_j$ and obtain solutions $\psi_j(t,x,z_j)$, $j=1,\ldots,N_{sc}$. With the solution ensembles $\{z_j, \psi_j(t,x,z_j)\}$, one then seeks to construct an approximation $\psi(t,x,z)$. Most constructions are linear and give an approximation of the form
\begin{equation}
\psi(t,x,z)=\sum_{j=1}^{N_{sc}}\psi_j(t,x,z_j)\ell_j(z)
\end{equation}
where the form of the function ${\ell_j(z)}$ depends on the construction method. For example, if the Lagrange interpolation is used, $\ell_j(z_k)=\delta_{jk}$. Other construction methods are also used like least-square regression, discrete projection, etc \cite{Xiu2009,Xiu2010}.

We now focus on the Lagrange interpolation. With the approximation constructed, one may be able to study the solution statistics. Take the expectation of the solution for example, it may be evaluated in the following way,
\begin{equation}
\mathbb{E}[\psi]=\sum_{j=1}^{N_{sc}}\psi_j\int\ell_j(z)d\mu(z).
\end{equation}

The time-splitting scheme \cite{Bao2002,Bao2003,Gosse I}, as the deterministic solver, ignores the additional structure provided by the periodic potential $V_{\Gamma}$. For the purpose of comparison, we present this method here.

\emph{Step 1.} In the first step, we solve the equation
\begin{equation}
i\varepsilon\partial_t\psi=-\frac{\epsilon^2}{2}\partial_{xx}\psi
\end{equation}
on a fixed time interval $\Delta t$, using the pseudospectral method.

\emph{Step 2.} In step 2, we solve the ODE
\begin{equation} \label{ts ode system}
i\varepsilon\partial_t\psi=\left(V_{\Gamma}(\frac{x}{\varepsilon})+U(x)\right)\psi
\end{equation}
on the same time interval, where the solution obtained in Step 1 serves as the initial condition for Step 2. It's easy to see that the solution of \eqref{ts ode system} is
\begin{equation}
\psi(t,x)=\psi(0,x)e^{-i(V_{\Gamma}(x/\varepsilon)+U(x))t/\varepsilon}.
\end{equation}

\section{Properties of the algorithm}
\label{sec 4}
\subsection{Conservation of mass}
Before we dig deeper into this topic, we should note that we have introduced randomness into this problem. So we must extend the concept of conservation of mass to stochastic setting \cite{Jin2015}.
\begin{defn} \label{defn of cons of mass}
Let $S$ be a numerical scheme for \eqref{Schrodinger}, which results in a solution $\psi_S(t,x,z)$. We say that $S$ has the property of strong conservation of mass if
\begin{equation} \label{strong sense}
\int_{\mathbb{R}}|\psi_S(t,x,z)|^2dx=\int_{\mathbb{R}}|\psi_S(0,x,z)|^2dx,\quad \forall t\geq0,\quad a.s.\quad z\in\Omega ;
\end{equation}
$S$ has the property of weak conservation of mass if
\begin{equation} \label{weak sense}
\mathbb{E}[\int_{\mathbb{R}}|\psi_S(t,x,z)|^2dx]=\mathbb{E}[\int_{\mathbb{R}}|\psi_S(0,x,z)|^2dx],\quad \forall t\geq0.
\end{equation}
where $\Omega$ is the random domain.
\end{defn}
In Definition~\ref{defn of cons of mass}, the integral in ~\eqref{strong sense} or ~\eqref{weak sense} w.r.t $x$ should be understood to be in the discrete sense.

It's obvious that strong conservation of mass is difficult to achieve as it almost requires the analytical solution of ~\eqref{Schrodinger} over the entire random domain. However, weak conservation of mass is more realistic.

As for the Bloch decomposition-based stochastic Galerkin method, since $\{\Phi_m(z)\}_{m=1}^{\infty}$ are orthogonal, if we take expectation and use Fubini Theorem, we have
\begin{equation}
\mathbb{E}[\int_{\mathbb{R}}|\psi_Q(t,x,z)|^2dx]=\int_{\mathcal{C}}|\vec{\psi}(t,x)|^2dx.
\end{equation}
So weak conservation of mass is equivalent to the conservation of the gPC coefficient vector, i.e.
\begin{equation}
\int_{\mathcal{C}}|\vec{\psi}(t,x)|^2dx=\int_{\mathcal{C}}|\vec{\psi}(0,x)|^2dx,\quad \forall t\geq0.
\end{equation}

We proceed to deduce that our algorithm has the property of weak conservation of mass.

In step 1, each component of $\vec{\psi}$ evolves separately. And by Proposition \ref{prop of 1d conservation}, the $L^2$-norm of each component of $\vec{\psi}$ is preserved during the time evolution if the assumption is satisfied.

In step 2, we solve the ODE system \eqref{ode system} analytically. Note that the coefficient matrix $A_U(x)$ is real and symmetric, which has $P$ real eigenvalues and $P$ orthonormal eigenvectors, the $L^2$-norm of $\vec{\psi}$ is preserved by the form of the analytic solution \eqref{ode solution}.

\begin{thm}
Under the same assumption of Proposition \ref{prop of 1d conservation}, the Bloch decomposition based stochastic Galerkin scheme has the property of weak conservation of mass and hence is unconditionally stable.
\end{thm}

\begin{rem}
The Schr\"{o}dinger equation is time reversible. And how $\psi$ evolves in both steps implies that our scheme is also time reversible, i.e., if we change $\Delta t$ into $-\Delta t$, we can reconstruct $\vec{\psi}^n$ with $\vec{\psi}^{n+1}$ as the initial data.
\end{rem}

\subsection{Conservation of energy}
For the one-dimensional Schr\"{o}dinger equation, the local energy density is
\begin{equation}
e(t,x)=\frac{\varepsilon^2}{2}|\partial_x \psi(t,x)|^2+\left(V_{\Gamma}(\frac{x}{\varepsilon})+U(x)\right)|\psi(t,x)|^2.
\end{equation}

In the stochastic setting, the energy density also depends on the random variable $z$, i.e.,
\begin{equation}
e(t,x,z)=\frac{\varepsilon^2}{2}|\partial_x \psi(t,x,z)|^2+\left(V_{\Gamma}(\frac{x}{\varepsilon})+U(x,z)\right)|\psi(t,x,z)|^2.
\end{equation}
Analogous to Definition~\ref{defn of cons of mass}, conservation of energy in the strong sense means
\begin{equation}
\int_{\mathbb{R}}e(t,x,z)dx=\int_{\mathbb{R}}e(0,x,z)dx,\quad \forall t\geq0,\quad a.s.\quad z\in\Omega ;
\end{equation}
conservation of energy in the weak sense means
\begin{equation}
\mathbb{E}[\int_{\mathbb{R}}e(t,x,z)dx]=\mathbb{E}[\int_{\mathbb{R}}e(0,x,z)dx],\quad \forall t\geq0.
\end{equation}
Again, we mainly focus on the weak conservation of energy. For our BD-SG scheme, take expectation and use Fubini's Theorem, the Hamiltonian or the energy that is desired to be conserved is
\begin{equation} \label{conservation of energy}
H(t)=\int_{\mathcal{C}}\frac{\varepsilon^2}{2}|\partial_x \vec{\psi}(t,x)|^2+V_{\Gamma}(\frac{x}{\varepsilon})|\vec{\psi}(t,x)|^2+\vec{\psi}(t,x)^HA_U(x)\vec{\psi}(t,x)dx,
\end{equation}
where $A_U$ is defined as before and the integral should also be considered on the discrete level.

Up to now, we still couldn't prove the weak conservation of energy for our scheme analytically. But our numerical results support this property.

\subsection{Numerical errors and Numerical costs}
Let us first consider truncation error. There're two types of truncation errors since we use gPC expansion up to a finite order and finitely many Bloch bands in the Bloch decomposition. According to the Cameron-Martin theorem \cite{Cameron1947}, the Fourier-Hermite series converge to any $L^2$ functional in the $L^2$ sense. Moreover, it's shown numerically by Xiu et al \cite{Xiu2002} that an appropriate choice of orthogonal polynomials can lead to exponential convergence of gPC expansions. What's more, the numerical experiments in \cite{Huang2007} show that the mass concentration in each Bloch band decays rapidly as $m\rightarrow +\infty$ if $\psi_{in}$ is smooth and thus only a few Bloch bands are needed to ensure sufficient accuracy, which indicates Bloch decomposition also achieves exponential convergence rate. So we just need use a few terms of both expansions since the temporal discretization error will eventually be dominant.

Certainly, the time-splitting scheme \eqref{step 1 of BDSG}-\eqref{ode system} is first order in time. But if we adapt the Strang's splitting procedure, we could obtain a second order scheme. This generalized algorithm still computes the dominant effects from dispersion and the periodic lattice potential in one step, maintaining the strong interaction, and treats the weak non-periodic random potential as a perturbation. Because the split-step error between the periodic and non-periodic parts is relatively small, our algorithm still preserves the advantage of allowing a relatively larger time step size \cite{Huang2007}.

It has been shown in \cite{Huang2007} that the complexity of Step 1 in the BD scheme is $O(MLR\log(R)) \lor O(RL\log(L))$ and that of Step 2 is $O(RL)$, which are comparable to that of the classical time-splitting scheme. It's obvious that the complexity of the BD-SG scheme is $O(PMLR\log(R)) \lor O(PRL\log(L)) \lor O(P^2RL)$. Numerical results will show that a moderate $P$ will be sufficient if the magnitude of randomness is not too large, which is consistent with the problem we're considering. So the complexity of the BD-SG scheme should be $O(PMLR\log(R)) \lor O(PRL\log(L))$ since $P\ll L \lor R$ regardless of the scale of $\varepsilon$ and note that $M=R$ if we want to ensure the weak conservation of mass. The above discussion indicates that the BD-SG scheme has a huge advantage in efficiency since the sampling methods combined with a deterministic solver, e.g., the classical time-splitting method, usually need a large number of samplings and thus a large number of calls of the deterministic solver.

\begin{rem}
The BD-SG algorithm \eqref{step 1 of BDSG}, \eqref{ode system} can be implemented in parallelization w.r.t. $P$ in a very natural way. Hence, given a moderate number of processors, the numerical cost on each processor would be comparable to the serial case where $P=1$. And the total computational time can be reduced dramatically.
\end{rem}

\section{Numerical experiments}
\label{sec 5}
In this section, we will present several numerical examples to illustrate the efficiency of our algorithm. For simplicity, we shall always assume a one-dimensional random variable $z$ obeying the uniform distribution on $[-1,1]$, and thus the Legendre polynomial chaos are adopted as the gPC basis. Multi-dimensional random variables can be handled in a similar and straightforward way. Typically, to examine the accuracy of numerical solutions, reference solutions are used, which are computed using very fine spatial grids, small time steps via the classical time-splitting method and the high-order stochastic collocation method.  And all the following experiments are conducted on a PC with 2.60 GHz CPU. To quantify the difference the difference between the numerical solution obtained via BD-SG, $\psi^{BG}(t,x,z)$, and the 'exact' solution, $\psi^{ex}(t,x,z)$, the following two metrics are introduced:
\begin{equation}
\Delta_{mean}^{BG}(t)=||\mathbb{E}[\psi^{ex}(t,\cdot,\cdot)]-\mathbb{E}[\psi^{BG}(t,\cdot,\cdot)]||_{L^2},
\end{equation}
\begin{equation}
\Delta_{den}^{BG}(t)=||\sqrt{\mathbb{E}[|\psi^{ex}(t,\cdot,\cdot)|^2]}-\sqrt{\mathbb{E}[|\psi^{BG}(t,\cdot,\cdot)|^2]}||_{L^2}.
\end{equation}
The former is the difference in the sense of mean, while the latter is in the sense of mean density.

We shall choose for \eqref{Schrodinger} the initial data $\psi_{in}\in L^2(\mathbb{R})$ of the form
\begin{equation}
\psi_{in}(x)=\left(\frac{10}{\pi}\right)^{1/4}e^{-5(x-\pi)^2}.
\end{equation}

For our numerical simulation below, we shall mainly use the following two types of periodic potentials, the Mathieu's model
\begin{equation} \label{mathieu}
V_{\Gamma}(x)=\cos(x)+1
\end{equation}
and the Kronig-Penny's model given by
\begin{equation} \label{KP}
V_{\Gamma}(x)=1=\sum_{\gamma\in\mathbb{Z}}\textbf{1}_{x\in [\frac{\pi}{2}+2\pi\gamma,\frac{3\pi}{2}+2\pi\gamma]}.
\end{equation}

With the above setting, we should turn to the random external potentials, namely, the following three types,
\begin{equation} \label{harmonic}
U(x)=|x-\pi|^2+0.5\big(z\cos(2x)+1\big),
\end{equation}
which is a harmonic potential with a weak noise, and then a non-smooth potential
\begin{equation} \label{step external}
U(x)=\textbf{1}_{x\in[\frac{\pi}{2},\frac{3\pi}{2}]}+2\frac{z+1}{x+1}.
\end{equation}
and
\begin{equation} \label{linear}
U(x,z)=(1+0.1z)x
\end{equation}
modelling a random (electric) force field.

In the first series of numerical experiments, we shall consider the convergence test w.r.t. the temporal and spatial discretization for our scheme. The results are given in Tables \ref{table 1}-\ref{table 4}.

\begin{table}[H]
\small
\caption{Convergence test w.r.t. $t$ for $V_{\Gamma}$ given by \eqref{mathieu}, $U$ given by \eqref{harmonic} at $T$}
\label{table 1}
\centering
\subtable[\scriptsize{$\varepsilon=1/4,T=1,\Delta x=\pi/128$ and gPC order of 4}]{
\begin{tabular}{|c|c|c|c|c|c|}
\hline
$\Delta t$ & 1/2 & 1/4 & 1/8 & 1/16 &1/32 \\ \hline
$\Delta_{mean}^{BG}$ & 1.36E-01 & 3.14E-02 & 7.70E-03 & 1.91E-03 & 4.78E-04 \\ \hline
convergence order & & 2.1 & 2.0 & 2.0 & 2.0 \\ \hline
$\Delta_{den}^{BG}$ & 1.16E-01 & 2.63E-02 & 6.42E-03 & 1.60E-03 & 3.99E-04 \\ \hline
convergence order & & 2.1 & 2.0 & 2.0 & 2.1 \\ \hline
\end{tabular}
}
\\
\subtable[\scriptsize{$\varepsilon=1/64,T=0.2,\Delta x=\pi/512$ and gPC order of 8}]{
\begin{tabular}{|c|c|c|c|c|c|}
\hline
$\Delta t$ & 1/10 & 1/20 & 1/40 & 1/80 &1/160 \\ \hline
$\Delta_{mean}^{BG}$ & 1.26E-01 & 1.53E-03 & 2.50E-04 & 6.22E-05 & 1.55E-05 \\ \hline
convergence order & & 3.0 & 2.6 & 2.0 & 2.0 \\ \hline
$\Delta_{den}^{BG}$ & 2.22E-02 & 1.97E-03 & 3.79E-04 & 9.40E-05 & 2.33E-05 \\ \hline
convergence order & & 3.5 & 2.4 & 2.0 & 2.0 \\ \hline
\end{tabular}
}
\\
\subtable[\scriptsize{$\varepsilon=1/512,T=0.02,\Delta x=\pi/16384$ and gPC order of 8}]{
\begin{tabular}{|c|c|c|c|c|c|}
\hline
$\Delta t$ & 1/50 & 1/100 & 1/200 & 1/400 &1/800 \\ \hline
$\Delta_{mean}^{BG}$ & 3.30E-03 & 1.03E-03 & 1.44E-04 & 3.13E-05 & 7.76E-06 \\ \hline
convergence order & & 1.7 & 2.8 & 2.2 & 2.0 \\ \hline
$\Delta_{den}^{BG}$ & 8.13E-03 & 2.65E-03 & 2.34E-04 & 5.68E-05 & 1.40E-05 \\ \hline
convergence order & & 1.6 & 3.5 & 2.0 & 2.0 \\ \hline
\end{tabular}
}
\end{table}

\begin{table}[H]
\small
\caption{Convergence test w.r.t. $t$ for $V_{\Gamma}$ given by \eqref{KP}, $U$ given by \eqref{harmonic} at $T$}
\label{table 2}
\centering
\subtable[\scriptsize{$\varepsilon=1/4,T=1,\Delta x=\pi/256$ and gPC order of 8}]{
\begin{tabular}{|c|c|c|c|c|c|}
\hline
$\Delta t$ & 1 & 1/2 & 1/4 & 1/8 &1/16 \\ \hline
$\Delta_{mean}^{BG}$ & 5.72E-01 & 1.05E-01 & 2.58E-02 & 6.32E-03 & 1.74E-03 \\ \hline
convergence order & & 2.4 & 2.0 & 2.0 & 1.9 \\ \hline
$\Delta_{den}^{BG}$ & 3.42E-01 & 6.88E-02 & 1.65E-02 & 4.06E-03 & 1.03E-03 \\ \hline
convergence order & & 2.3 & 2.1 & 2.0 & 2.0 \\ \hline
\end{tabular}
}
\\
\subtable[\scriptsize{$\varepsilon=1/16,T=0.5,\Delta x=\pi/512$ and gPC order of 8}]{
\begin{tabular}{|c|c|c|c|c|c|}
\hline
$\Delta t$ & 1/2 & 1/4 & 1/8 & 1/16 &1/32 \\ \hline
$\Delta_{mean}^{BG}$ & 1.21E-01 & 3.08E-02 & 7.02E-03 & 1.96E-03 & 5.79E-04 \\ \hline
convergence order & & 2.0 & 2.1 & 1.8 & 1.8 \\ \hline
$\Delta_{den}^{BG}$ & 1.54E-01 & 3.27E-02 & 6.90E-03 & 1.72E-03 & 4.51E-04 \\ \hline
convergence order & & 2.2 & 2.2 & 2.0 & 1.9 \\ \hline
\end{tabular}
}
\end{table}

\begin{table}[H]
\small
\caption{Convergence test w.r.t. $x$ for $V_{\Gamma}$ given by \eqref{mathieu}, $U$ given by \eqref{step external} at $T$}
\label{table 3}
\centering
\subtable[\scriptsize{$\varepsilon=1/512,T=0.02,\Delta t=1/500$ and gPC order of 8}]{
\begin{tabular}{|c|c|c|c|c|c|}
\hline
$\Delta x$ & $\pi/512$ & $\pi/1024$ & $\pi/2048$ & $\pi/4096$ \\ \hline
$\Delta_{mean}^{BG}$ & 2.45E+01 & 4.21E-01 & 4.72E-03 & 3.29E-06 \\ \hline
convergence order & & 5.9 & 6.5 & 10.5 \\ \hline
$\Delta_{den}^{BG}$ & 1.28E+02 & 1.99E+00 & 1.95E-02 & 9.51E-06 \\ \hline
convergence order & & 6.0 & 6.7 & 11.0 \\ \hline
\end{tabular}
}
\\
\subtable[\scriptsize{$\varepsilon=1/1024,T=0.01,\Delta t=1/1000$ and gPC order of 8}]{
\begin{tabular}{|c|c|c|c|c|c|}
\hline
$\Delta x$ & $\pi/1024$ & $\pi/2048$ & $\pi/4096$ & $\pi/8192$ \\ \hline
$\Delta_{mean}^{BG}$ & 2.45E+01 & 4.12E-01 & 4.83E-03 & 4.24E-06 \\ \hline
convergence order & & 5.9 & 6.4 & 10.2 \\ \hline
$\Delta_{den}^{BG}$ & 1.28E+02 & 2.01E+00 & 2.00E-02 & 5.58E-06 \\ \hline
convergence order & & 6.0 & 6.6 & 11.8 \\ \hline
\end{tabular}
}
\end{table}

\begin{table}[H]
\small
\caption{Convergence test w.r.t. $x$ for $V_{\Gamma}$ given by \eqref{KP}, $U$ given by \eqref{step external} at $T$}
\label{table 4}
\centering
\subtable[\scriptsize{$\varepsilon=1/64,T=0.05,\Delta t=1/200$ and gPC order of 8}]{
\begin{tabular}{|c|c|c|c|c|c|}
\hline
$\Delta x$ & $\pi/64$ & $\pi/128$ & $\pi/256$ & $\pi/512$ \\ \hline
$\Delta_{mean}^{BG}$ & 7.58E+00 & 4.71E+00 & 3.57E-01 & 1.82E-02 \\ \hline
convergence order & & 0.7 & 3.7 & 4.3 \\ \hline
$\Delta_{den}^{BG}$ & 1.16E+01 & 6.56E+00 & 6.17E-01 & 3.67E-02 \\ \hline
convergence order & & 0.8 & 3.4 & 4.1 \\ \hline
\end{tabular}
}
\\
\subtable[\scriptsize{$\varepsilon=1/1024,T=0.01,\Delta t=1/600$ and gPC order of 8}]{
\begin{tabular}{|c|c|c|c|c|c|}
\hline
$\Delta x$ & $\pi/1024$ & $\pi/2048$ & $\pi/4096$ & $\pi/8192$ \\ \hline
$\Delta_{mean}^{BG}$ & 6.07E-01 & 4.80E-01 & 5.81E-02 & 1.98E-03 \\ \hline
convergence order & & 0.3 & 3.0 & 4.9 \\ \hline
$\Delta_{den}^{BG}$ & 3.00E+00 & 1.90E+00 & 1.88E-01 & 5.49E-03 \\ \hline
convergence order & & 0.7 & 3.3 & 5.1 \\ \hline
\end{tabular}
}
\end{table}

The results in Tables \ref{table 1} and \ref{table 2} show that the BD-SG scheme is second order in time. And the results in Tables \ref{table 3} and \ref{table 4} show that the BD-SG scheme exhibits exponential convergence in space in both cases of periodic potentials while in the case of a smooth periodic potential, the convergence rate in space is even faster.

In Figures \ref{Fig. 1} and \ref{Fig. 2}, results of convergence test w.r.t. gPC order are shown. We can see that both errors will saturate quickly at a certain gPC order as the temporal or spatial discretization error would become dominant. And before saturation, the fast exponential convergence w.r.t. the order of gPC expansion can be observed. Therefore, our BD-SG scheme just has to utilize a few gPC order to obtain an accurate solution and has the advantage of lower computation costs.

\begin{figure}[H]
\centering
\subfigure[\scriptsize{$\varepsilon=1/4,T=1,\Delta t=0.01,\Delta x=\pi/128$}]{
\includegraphics[width=0.6\textwidth]{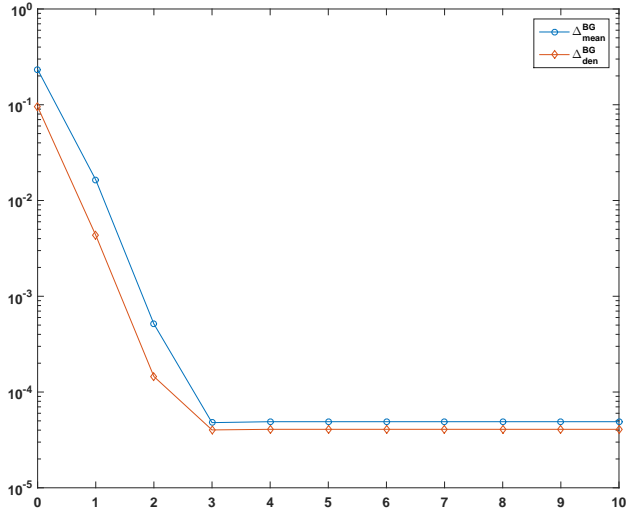}}
\subfigure[\scriptsize{$\varepsilon=1/256,T=0.05,\Delta t=0.0005,\Delta x=\pi/4096$}]{
\includegraphics[width=0.6\textwidth]{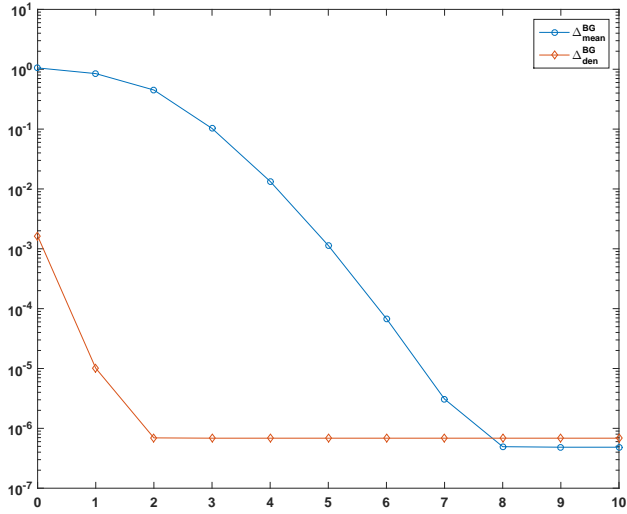}}
\caption{Convergence test w.r.t. gPC order for $V_{\Gamma}$ given by \eqref{mathieu}, $U$ given by \eqref{harmonic} at $T$}
\label{Fig. 1}
\end{figure}

\begin{figure}[H]
\centering
\subfigure[\scriptsize{$\varepsilon=1/4,T=1,\Delta t=0.1,\Delta x=\pi/256$}]{
\includegraphics[width=0.6\textwidth]{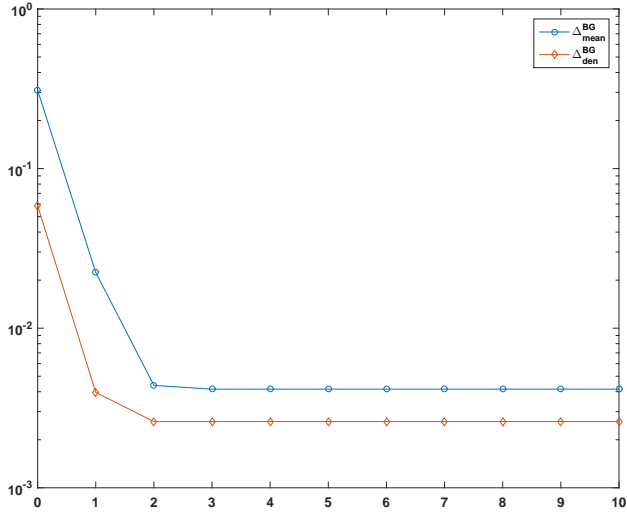}}
\subfigure[\scriptsize{$\varepsilon=1/64,T=0.22,\Delta t=0.0025,\Delta x=\pi/4096$}]{
\includegraphics[width=0.6\textwidth]{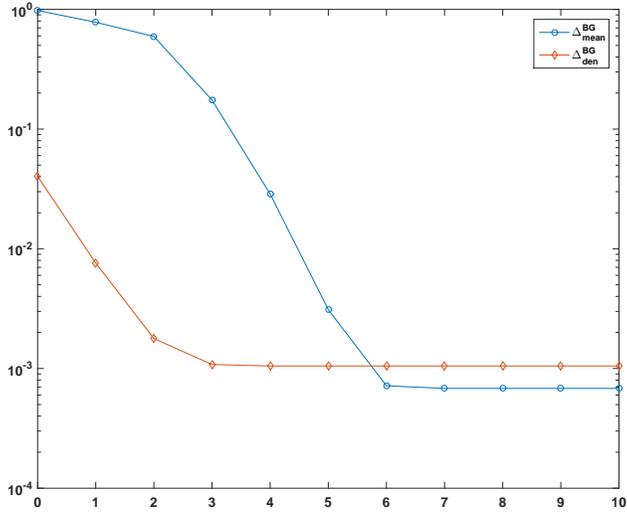}}
\caption{Convergence test w.r.t. gPC order for $V_{\Gamma}$ given by \eqref{KP}, $U$ given by \eqref{harmonic} at $T$}
\label{Fig. 2}
\end{figure}

Another natural idea to solve \eqref{Schrodinger} is to combine the classical time-splitting method and Monte Carlo method (TS-MC). In general, the convergence rate of Monte Carlo method is pretty slow at $O(1/\sqrt{K})$, where $K$ is the number of realizations. So we expect that the BD-SG scheme works better than TS-MC. We apply both methods to the example of $V_{\Gamma}$ given by \eqref{mathieu} and $U$ given by \eqref{linear} with $\varepsilon=\frac{1}{4}$. We define $\Delta^{TM}_{mean}$ and $\Delta^{TM}_{den}$ accordingly. The results are shown in Figure \ref{Fig. 3} and Table \ref{table 5}.

\begin{figure}[H]
\centering
\subfigure[\scriptsize{$|\psi^{ex}(T,x)|^2$}]{
\includegraphics[width=0.3\textwidth]{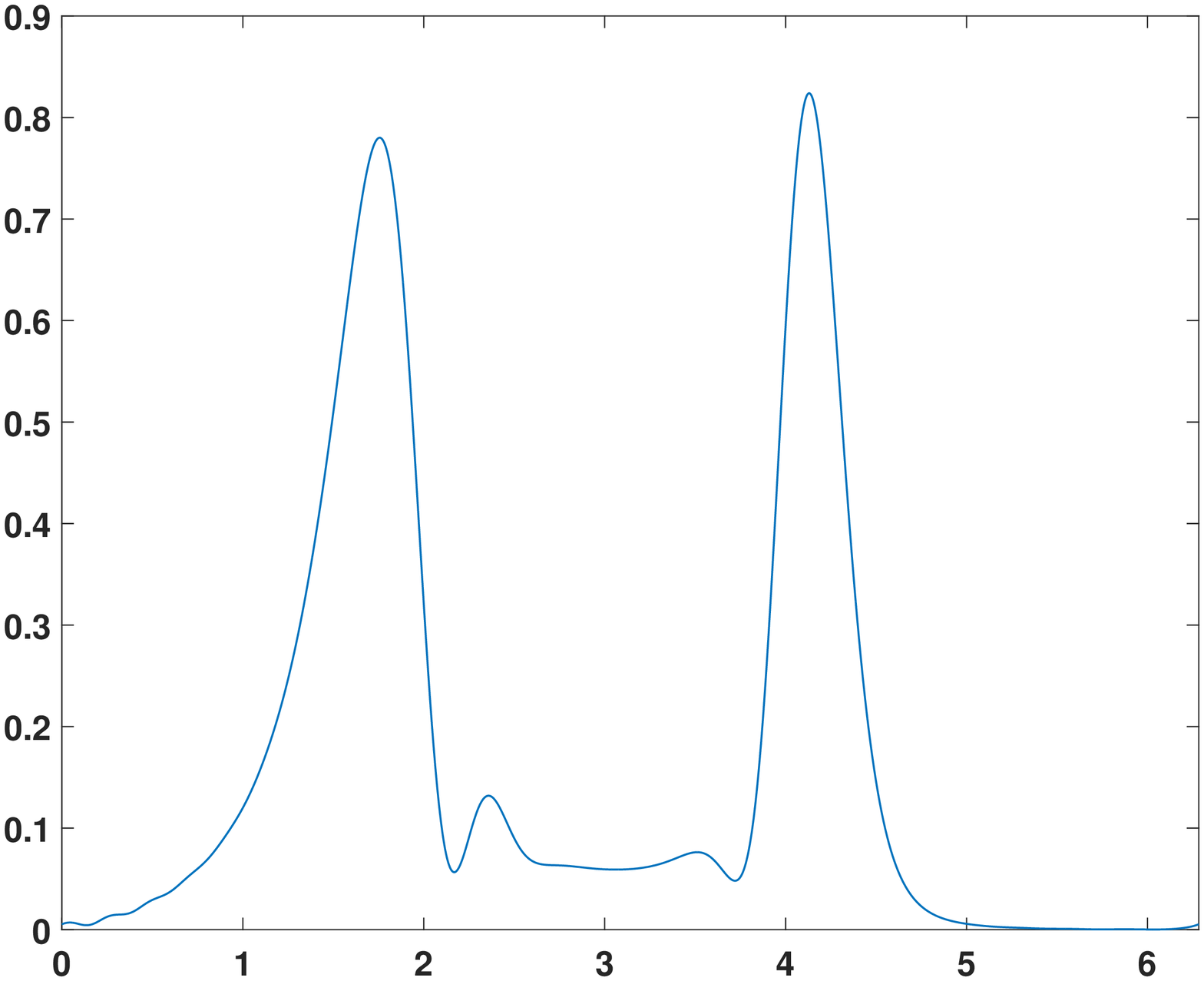}}
\subfigure[\scriptsize{$|\mathbb{E}[\psi^{ex}-\psi^{BG}]|$}]{
\includegraphics[width=0.3\textwidth]{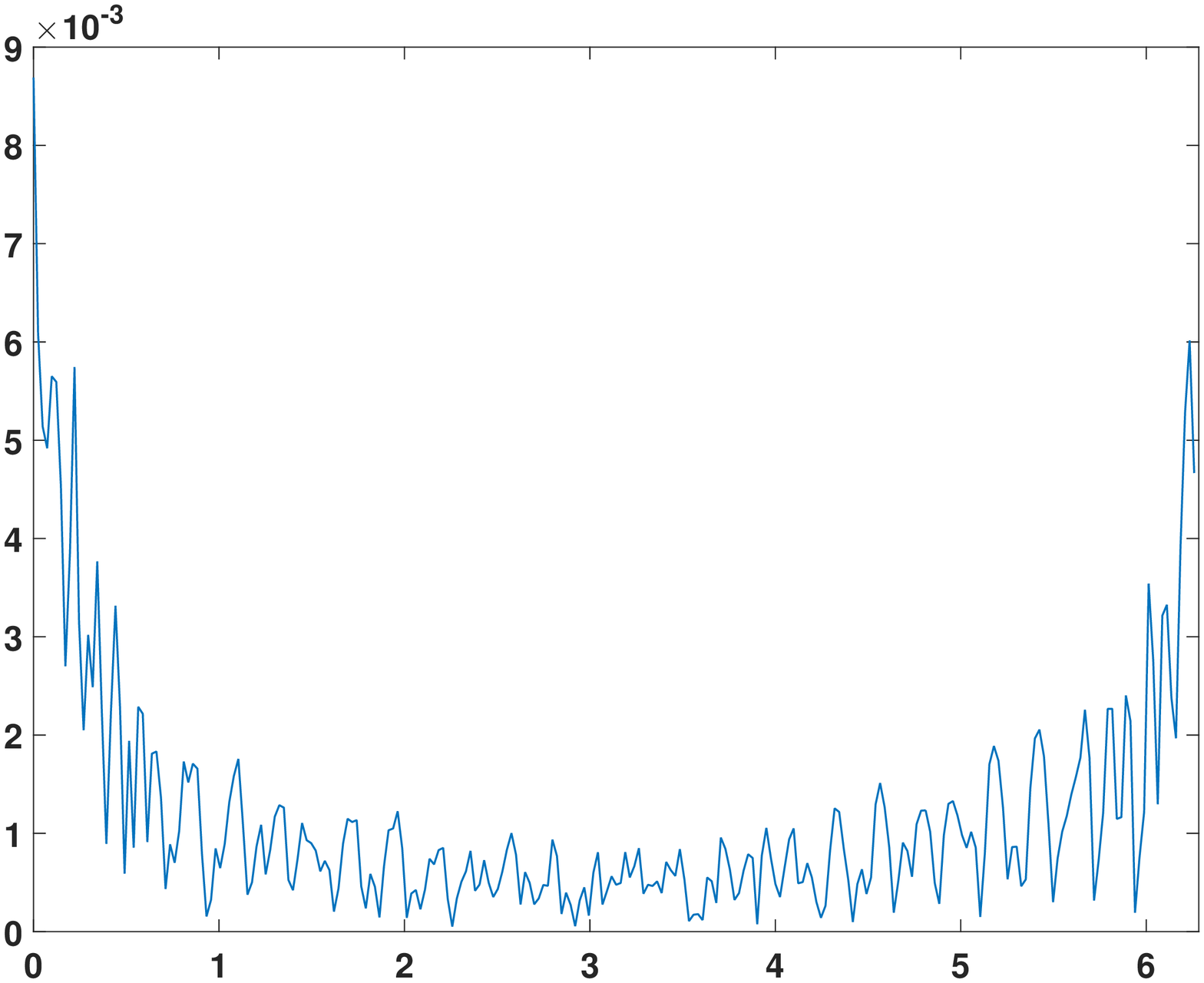}}
\subfigure[\scriptsize{$|\mathbb{E}[\psi^{ex}-\psi^{TM}]|$}]{
\includegraphics[width=0.3\textwidth]{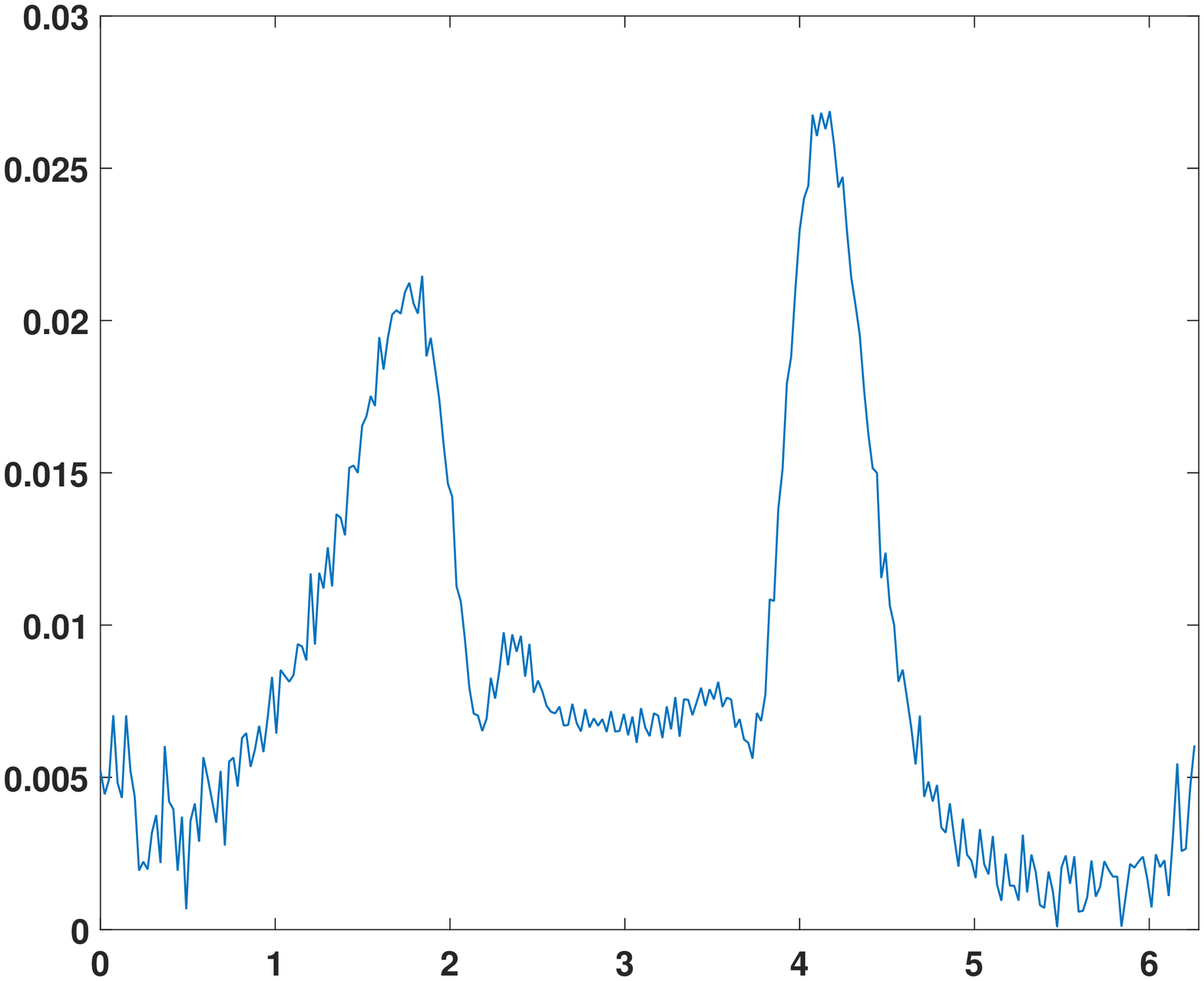}}
\caption{$V_{\Gamma}$ given by \eqref{mathieu}, $U$ given by \eqref{linear} at $T=1$ with $\varepsilon=\frac{1}{4}$. $\Delta t=1/100,\Delta x=\pi/128$ for both methods. BD-SG uses gPC order of 4 while TS-MC uses 1,000 realizations}
\label{Fig. 3}
\end{figure}

\begin{table}[H]
\small
\centering
\caption{TS-MC for $V_{\Gamma}$ given by \eqref{mathieu} and $U$ given by \eqref{linear} with $\varepsilon=1/4,T=1,\Delta t=1/100,\Delta x=\pi/128$ and $K$ realizations while $\Delta^{BG}_{mean}=4.25E-03$ and $\Delta^{BG}_{den}=2.88E-03$ for BD-SG with the same temporal and spatial discretization and gPC order of 4.}
\label{table 5}
\begin{tabular}{|c|c|c|c|c|}
\hline
$K$ & 10 & 100 & 1,000 & 10,000 \\ \hline
$\Delta^{TM}_{mean}$ & 1.36E-01 & 4.61E-02 & 1.14E-02 & 8.91E-03 \\ \hline
$\Delta^{TM}_{den}$ & 5.87E-03 & 2.93E-03 & 2.84E-03 & 2.84E-03 \\ \hline
\end{tabular}
\end{table}

Compare the results computed by TS-MC with that computed by BD-SG, it's clear that, with the same temporal and spatial discretization, the TS-MC needs around 1,000 realizations to achieve the same mean density error level as BD-SG scheme with gPC order of 4 and more than 10,000 realizations to achieve the same mean error level. With the same temporal and spatial discretization, the CPU time is around 3 seconds for BD-SG using gPC order of 4 in ignorance of the preparatory step while it's 5 seconds for TS-MC using 1,000 realizations and 51 seconds for TS-MC using 10,000 realizations.  Thus, it's easy to see that BD-SG works much better than TS-MC.

\begin{rem}
In \cite{Huang2007}, it's numerically shown that BD method works much better than TS method in the case where $V_{\Gamma}$ is non-smooth and/or $\varepsilon\ll 1$. So we may expect that the BD-SG method works much better than TS-MC method in those cases.
\end{rem}

Next, we want to compare the BD-SG method and the classical time-splitting method with stochastic collocation method (TS-SC). Since the stochastic collocation method makes use of the polynomial approximation theory, it also works better than the Monte Carlo method. We apply BD-SG and TS-SC to 2 cases. Again, $\Delta^{TC}_{mean}$ and $\Delta^{TC}_{den}$ are defined accordingly. The results are shown in Figures \ref{Fig. 4}, \ref{Fig. 5} and Tables \ref{table 6}, \ref{table 7}.

\begin{figure}[H]
\centering
\subfigure[\scriptsize{$|\psi^{ex}(T,x)|^2$}]{
\includegraphics[width=0.3\textwidth]{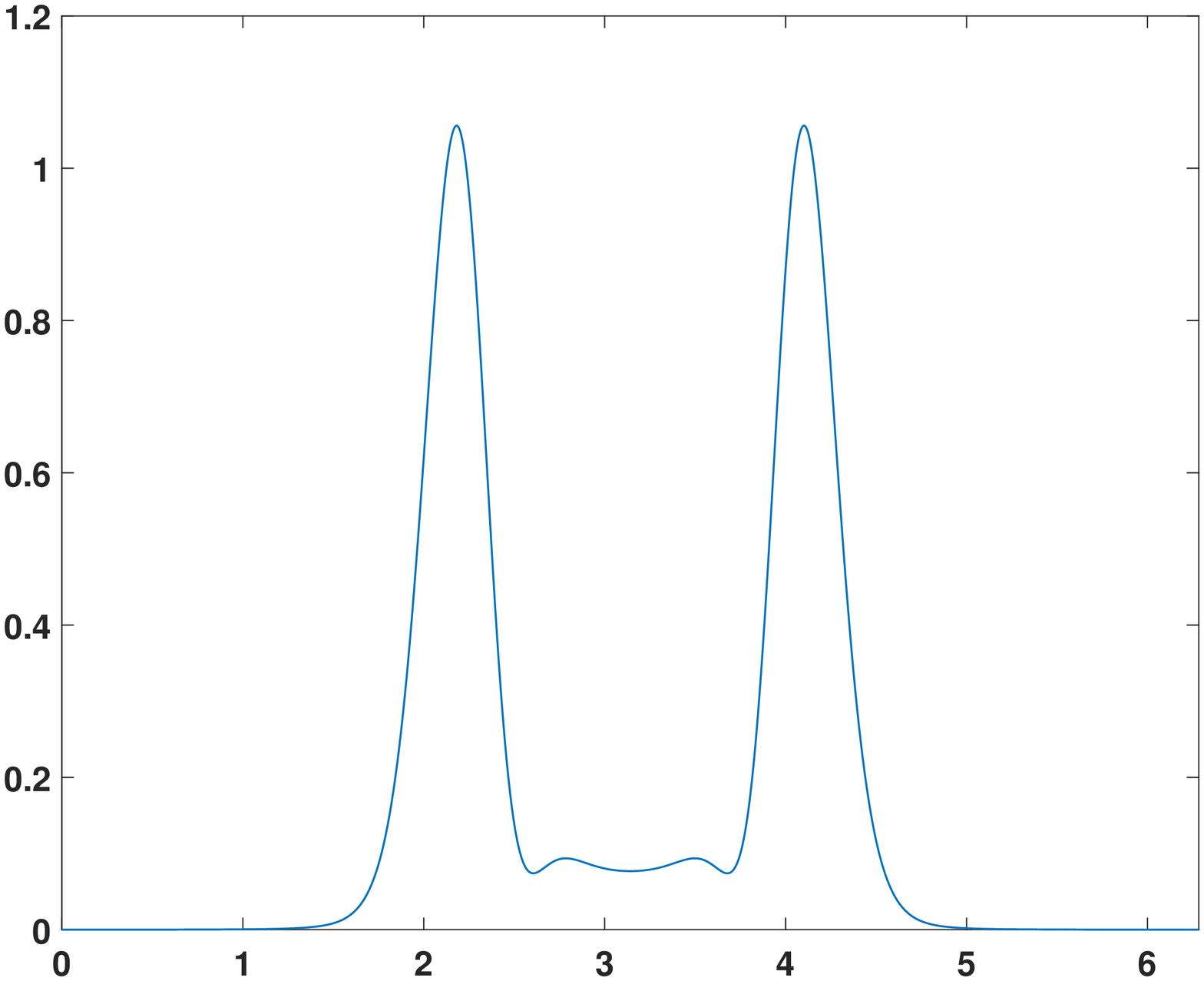}}
\subfigure[\scriptsize{$|\mathbb{E}[\psi^{ex}-\psi^{BG}]|$}]{
\includegraphics[width=0.3\textwidth]{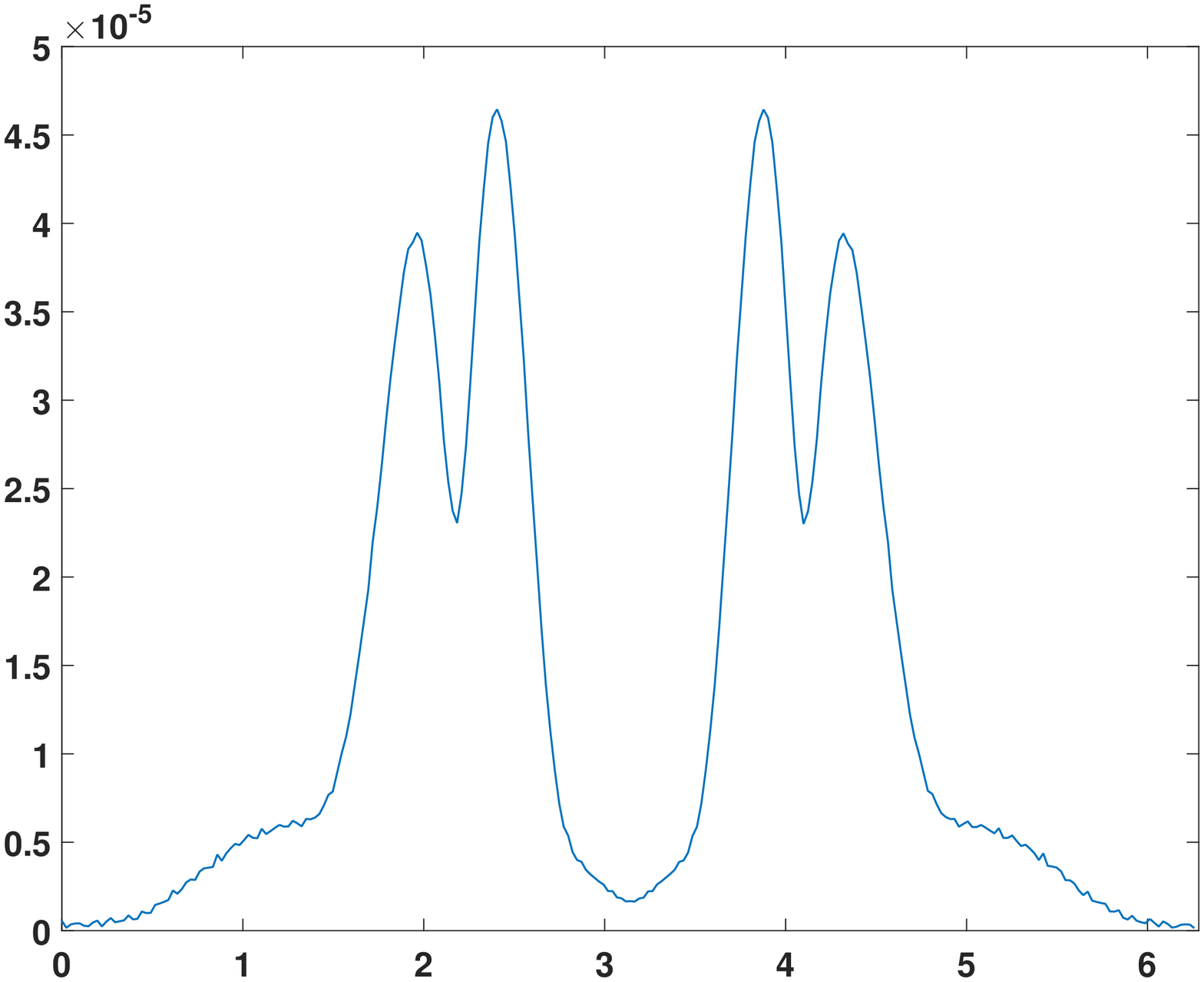}}
\subfigure[\scriptsize{$|\mathbb{E}[\psi^{ex}-\psi^{TC}]|$}]{
\includegraphics[width=0.3\textwidth]{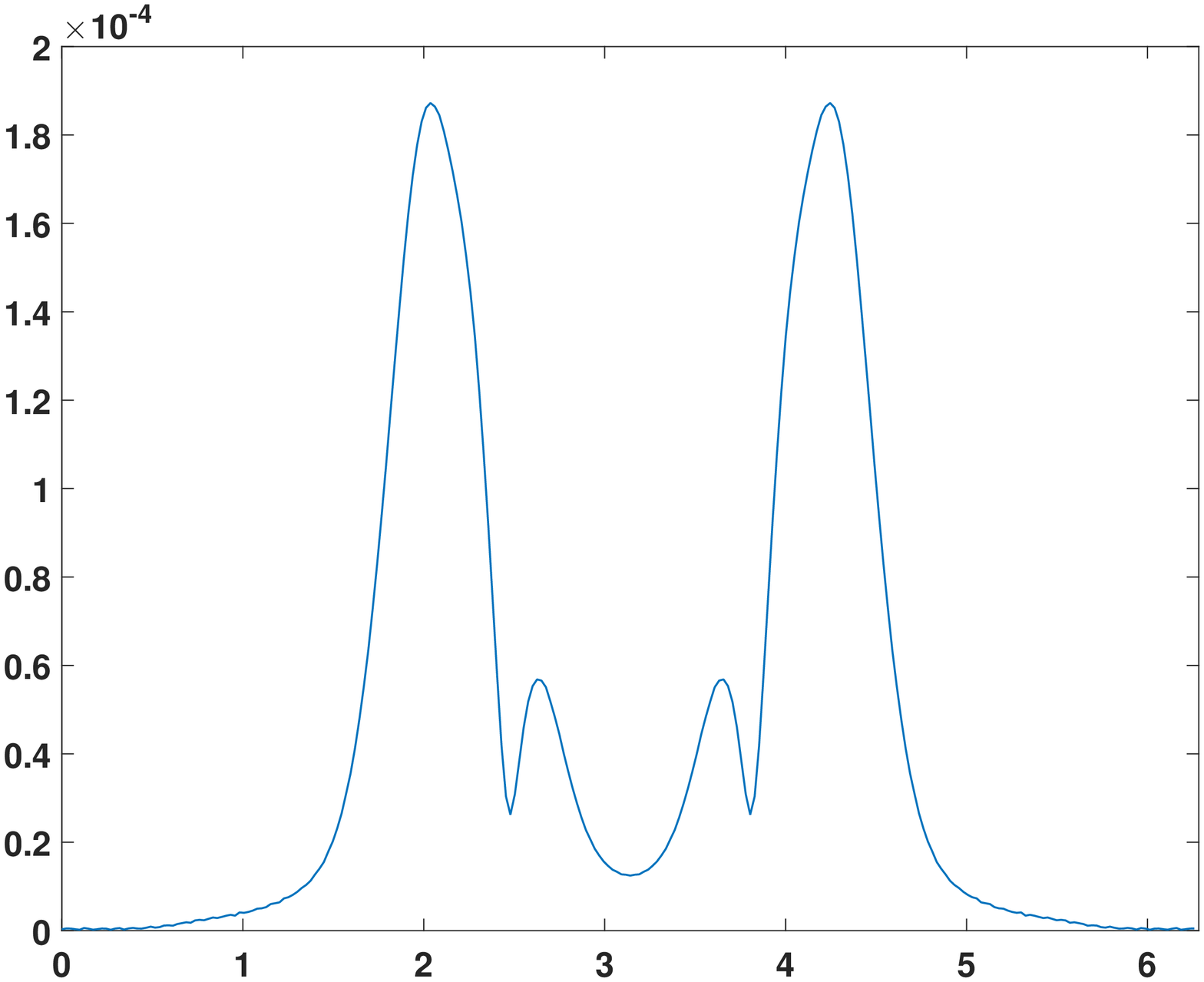}}
\caption{$V_{\Gamma}$ given by \eqref{mathieu}, $U$ given by \eqref{harmonic} at $T=1$ with $\varepsilon=\frac{1}{4}$. $\Delta t=1/100,\Delta x=\pi/128$ for both methods. BD-SG uses gPC order of 4 while TS-SC uses 5 quadrature nodes}
\label{Fig. 4}
\end{figure}

\begin{table}[H]
\small
\caption{Comparison between BD-SG and TS-SC for $V_{\Gamma}$ given by \eqref{mathieu}, $U$ given by \eqref{harmonic} at $T$}
\label{table 6}
\centering
\subtable[\scriptsize{BD-SG}]{
\begin{tabular}{|c|c|c|c|c|c|c|}
  \hline
  $\varepsilon$ & $T$ & $\Delta t$ & $\Delta x$ & gPC order & $\Delta_{mean}^{BG}$ & $\Delta_{den}^{BG}$ \\ \hline
  1/4 & 1 & 0.01 & $\pi/128$ & 4 & 4.90E-05 & 4.08E-05 \\ \hline
  1/1024 & 0.01 & 0.001 & $\pi/8192$ & 4 & 1.96E-03 & 1.83E-05 \\
  \hline
\end{tabular}
}
\\
\subtable[\scriptsize{TS-SC}]{
\begin{tabular}{|c|c|c|c|c|c|c|}
  \hline
  $\varepsilon$ & $T$ & $\Delta t$ & $\Delta x$ & number of nodes & $\Delta_{mean}^{TC}$ & $\Delta_{den}^{TC}$ \\ \hline
  1/4 & 1 & 0.01 & $\pi/128$ & 5 & 1.82E-04 & 8.80E-05 \\ \hline
  1/1024 & 0.01 & 0.000004 & $\pi/8192$ & 5 & 1.96E-03 & 4.40E-05 \\
  \hline
\end{tabular}
}
\end{table}

\begin{figure}[H]
\centering
\subfigure[\scriptsize{$|\psi^{ex}(T,x)|^2$}]{
\includegraphics[width=0.3\textwidth]{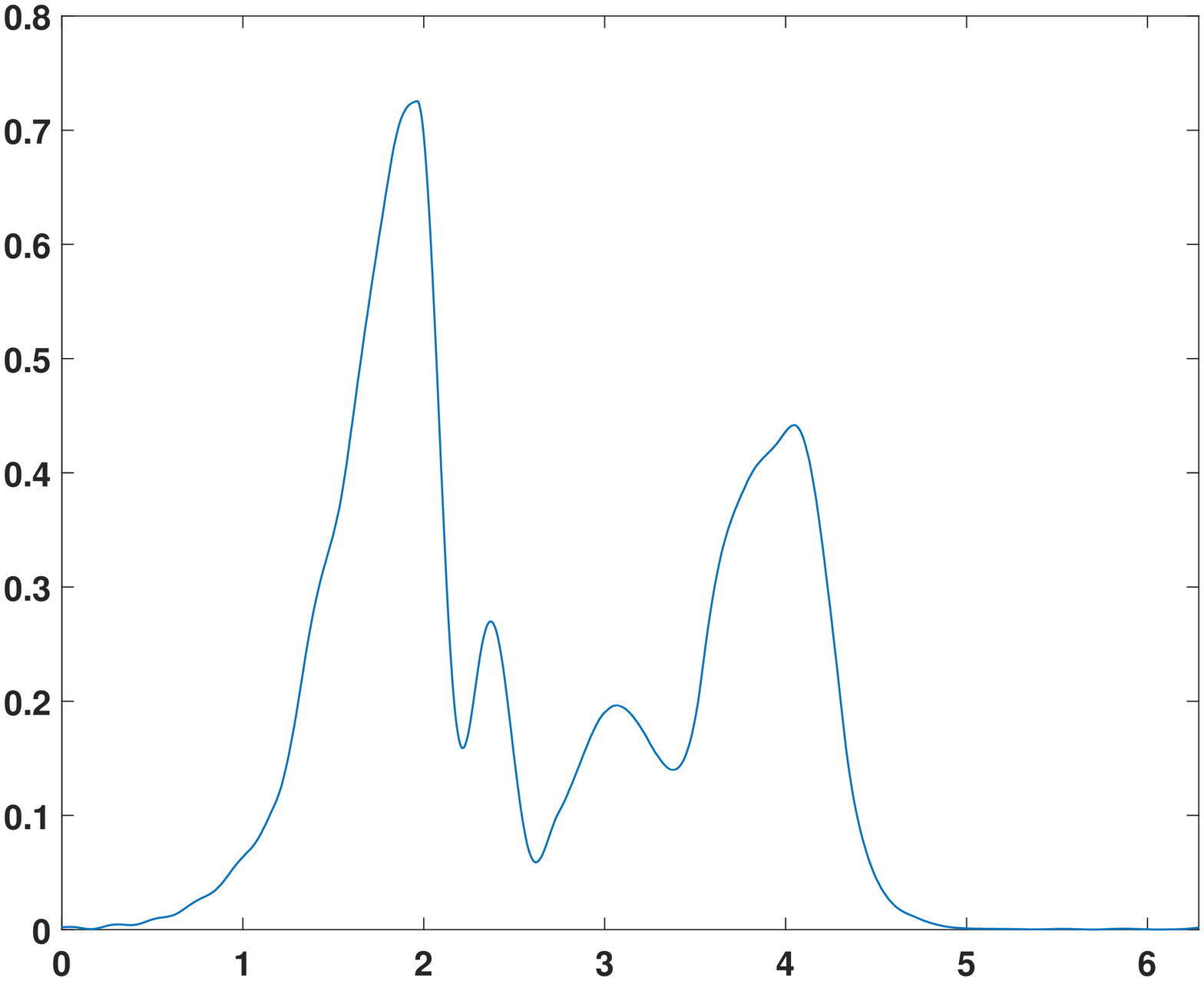}}
\subfigure[\scriptsize{$|\mathbb{E}[\psi^{ex}-\psi^{BG}]|$}]{
\includegraphics[width=0.3\textwidth]{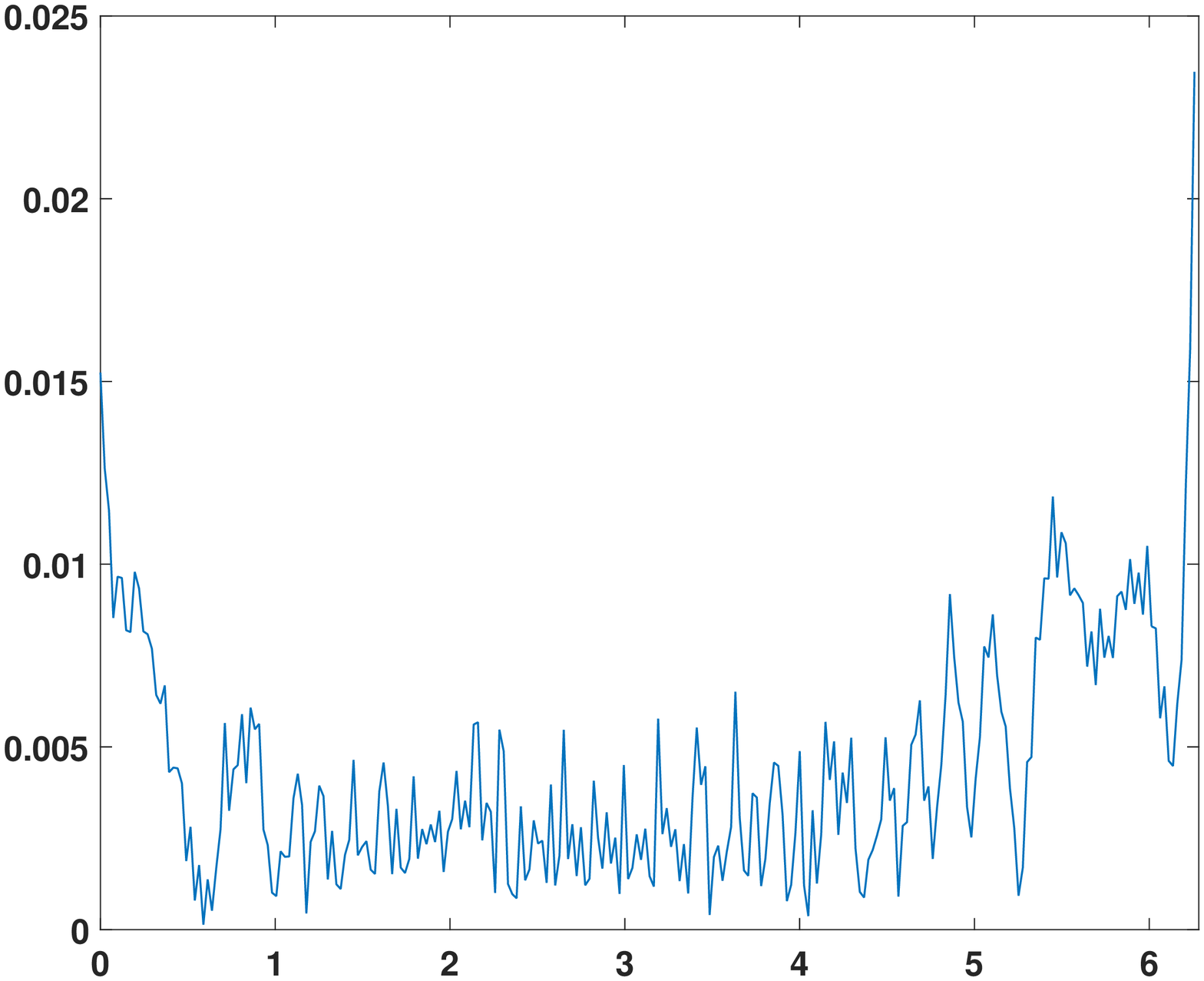}}
\subfigure[\scriptsize{$|\mathbb{E}[\psi^{ex}-\psi^{TC}]|$}]{
\includegraphics[width=0.3\textwidth]{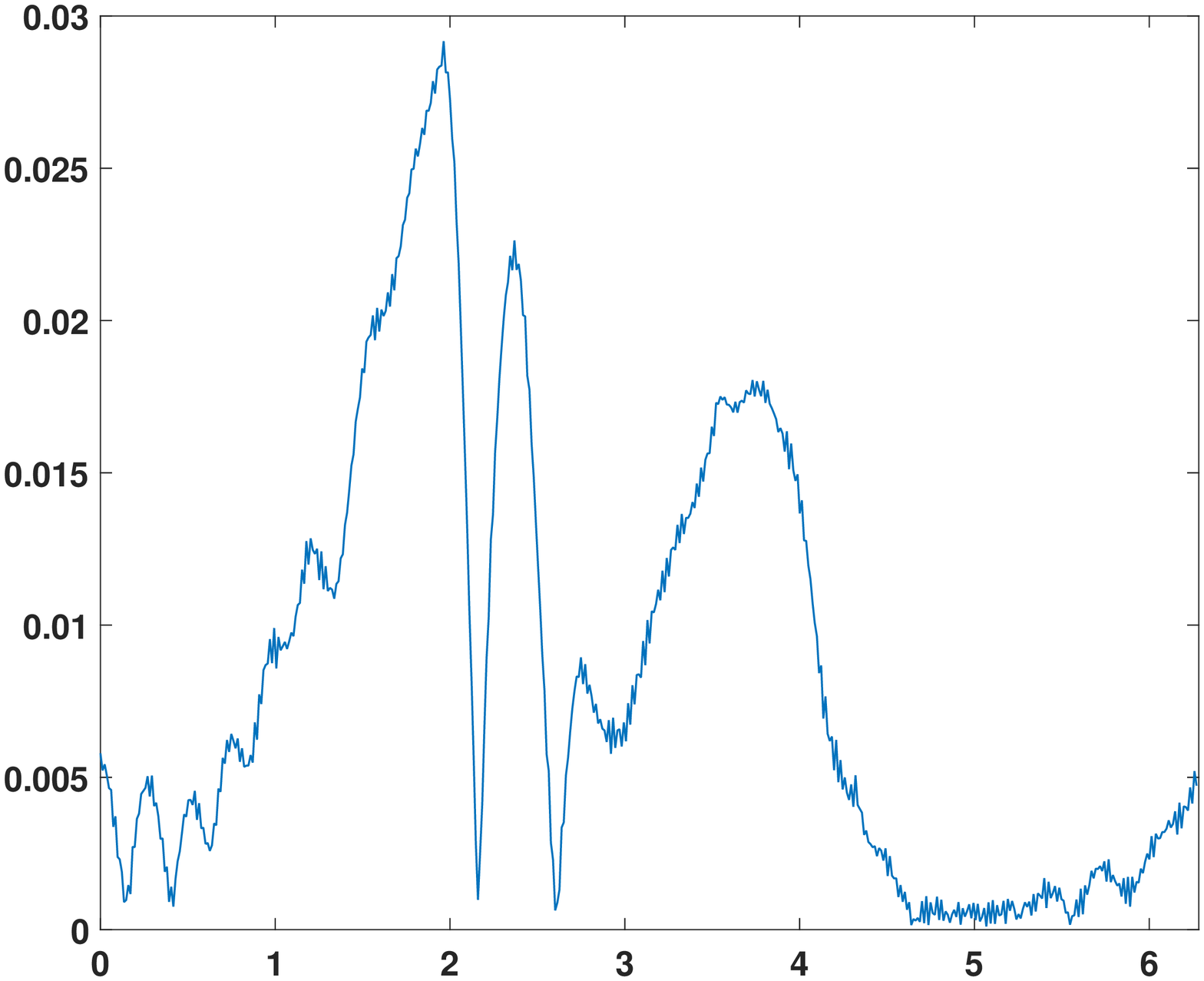}}
\caption{$V_{\Gamma}$ given by \eqref{KP}, $U$ given by \eqref{linear} at $T=1$ with $\varepsilon=\frac{1}{4}$. $\Delta t=1/10,\Delta x=\pi/128$ for BD-SG and $\Delta t=1/1000,\Delta x=\pi/256$ for BD-SG. BD-SG uses gPC order of 4 while TS-SC uses 5 quadrature nodes}
\label{Fig. 5}
\end{figure}

\begin{table}[H]
\small
\caption{Comparison between BD-SG and TS-SC for $V_{\Gamma}$ given by \eqref{KP}, $U$ given by \eqref{linear} at $T$}
\label{table 7}
\centering
\subtable[\scriptsize{BD-SG}]{
\begin{tabular}{|c|c|c|c|c|c|c|}
  \hline
  $\varepsilon$ & $T$ & $\Delta t$ & $\Delta x$ & gPC order & $\Delta_{mean}^{BG}$ & $\Delta_{den}^{BG}$ \\ \hline
  1/4 & 1 & 0.1 & $\pi/128$ & 4 & 1.35E-02 & 9.03E-03 \\ \hline
  1/1024 & 0.01 & 0.001 & $\pi/8192$ & 4 & 3.44E-03 & 1.74E-02 \\
  \hline
\end{tabular}
}
\\
\subtable[\scriptsize{TS-SC}]{
\begin{tabular}{|c|c|c|c|c|c|c|}
  \hline
  $\varepsilon$ & $T$ & $\Delta t$ & $\Delta x$ & number of nodes & $\Delta_{mean}^{TC}$ & $\Delta_{den}^{TC}$ \\ \hline
  1/4 & 1 & 0.001 & $\pi/256$ & 5 & 2.83E-02 & 1.17E-02 \\ \hline
  1/1024 & 0.01 & 0.00005 & $\pi/32768$ & 5 & 3.69E-03 & 2.20E-02 \\
  \hline
\end{tabular}
}
\end{table}

Comparison between BD-SG and TS-SC tells us that when $V_{\Gamma}$ is smooth, these two methods give almost the same result when $\varepsilon=O(1)$ but if $\varepsilon\ll 1$, the BD-SG method can achieve an accurate solution with $\Delta t=O(1)$, $\Delta x=O(\varepsilon)$ while the TS-SC needs finer time steps to get a solution at the same error level. When $V_{\Gamma}$ is non-smooth, the BD-SG achieves better results than TS-SC even if $\varepsilon=O(1)$ and as $\varepsilon$ gets smaller, the advantage is more prominent. Thus, the BD-SG method still preserves the advantage of allowing relatively larger time step size when $V_{\Gamma}$ is non-smooth and/or $\varepsilon \ll 1$. The result is consistent with that given in \cite{Huang2007}.

\begin{rem}
As for the disposal against randomness, it's stated in \cite{Xiu2009} that the exact cost comparison between gPC Galerkin and stochastic collocation depends on many factors, a large number of which are unknown, but it's fair to say that the gPC Galerkin should be preferred if the coupling of gPC Galerkin does not incur much additional computational cost or efficient solvers can be developed to decouple the gPC system, especially for the case where the dimension of random space is high. And in \cite{Back2011}, it's shown that the numerical performances of these two methods are comparable. So we don't present examples of direct comparison between these two method.
\end{rem}

\begin{rem}
Apparently, the finite difference method can be chosen as the deterministic solve for the Schr\"{o}dinger equation. But in many references, e.g. \cite{Bao2002}, it has been shown that the classical time-splitting method works much better than the finite difference method. So we don't generalize the finite difference method to the stochastic setting as a comparison with our BD-SG method.
\end{rem}

Thus, by comparing the BD-SG method with some other methods, it can be concluded that the combination of the Bloch decomposition and the stochastic Galerkin method is an appropriate choice for the Schr\"{o}dinger equation with a periodic potential and a random external potential.

\subsection{Conservation quantities}
Recall that the total mass and energy are two of the most important conservation quantities for the Schr\"{o}dinger equation. And we have discussed the conservation of mass and energy in the weak sense in Section \ref{sec 4}. We prove analytically that our BD-SG scheme enjoys the property of weak conservation of mass. In this subsection, we will numerically show that the BD-SG scheme enjoys both properties of weak conservation of mass and energy. We define
\begin{equation}
M(t)=\mathbb{E}[\int_{\mathbb{R}}|\psi(t,x,z)|^2dx]
\end{equation}
and recall that in our case, the energy we want to preserve is given by \eqref{conservation of energy}. Figures \ref{Fig. 6} and \ref{Fig. 7} show that the performance of conservation of mass is better than that of energy. But the errors of both quantities are acceptable.

\begin{figure}[H]
\centering
\subfigure[\scriptsize{$M(t)/M(0)$}]{
\includegraphics[width=0.4\textwidth]{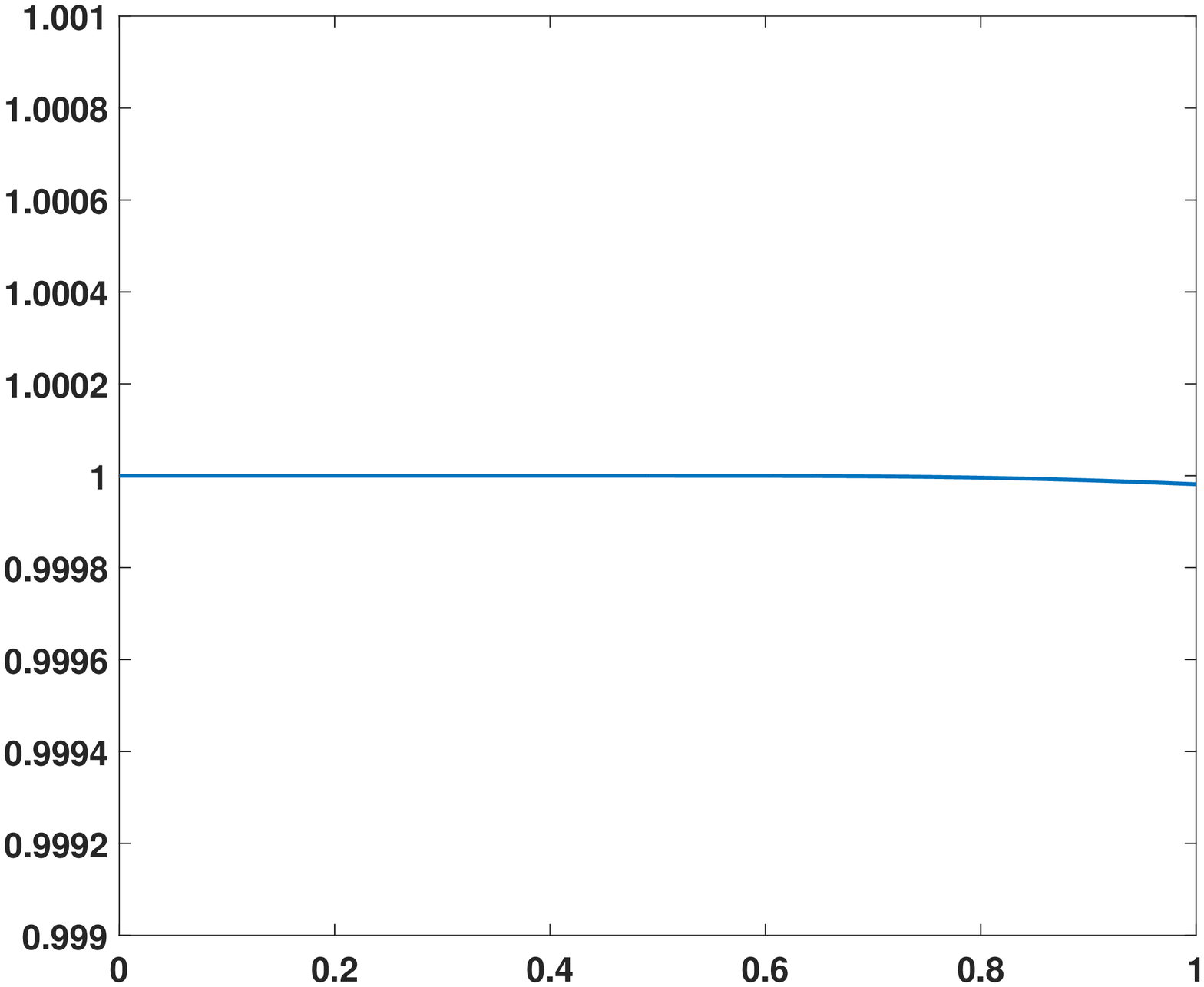}}
\subfigure[\scriptsize{$H(t)/H(0)$}]{
\includegraphics[width=0.4\textwidth]{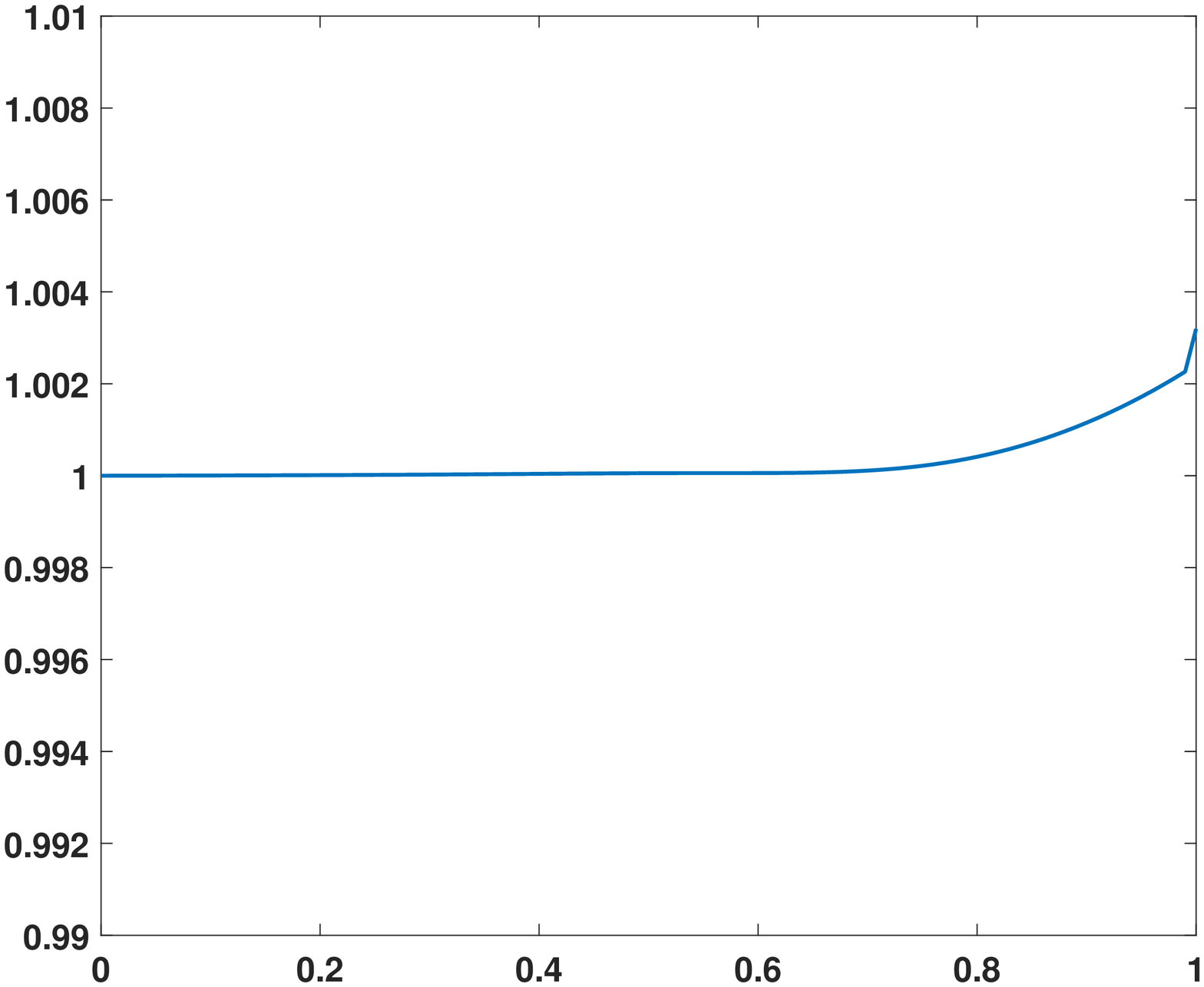}}
\caption{$V_{\Gamma}$ given by \eqref{mathieu}, $U$ given by \eqref{step external} and $\varepsilon=1/4$}
\label{Fig. 6}
\end{figure}

\begin{figure}[H]
\centering
\subfigure[\scriptsize{$M(t)/M(0)$}]{
\includegraphics[width=0.4\textwidth]{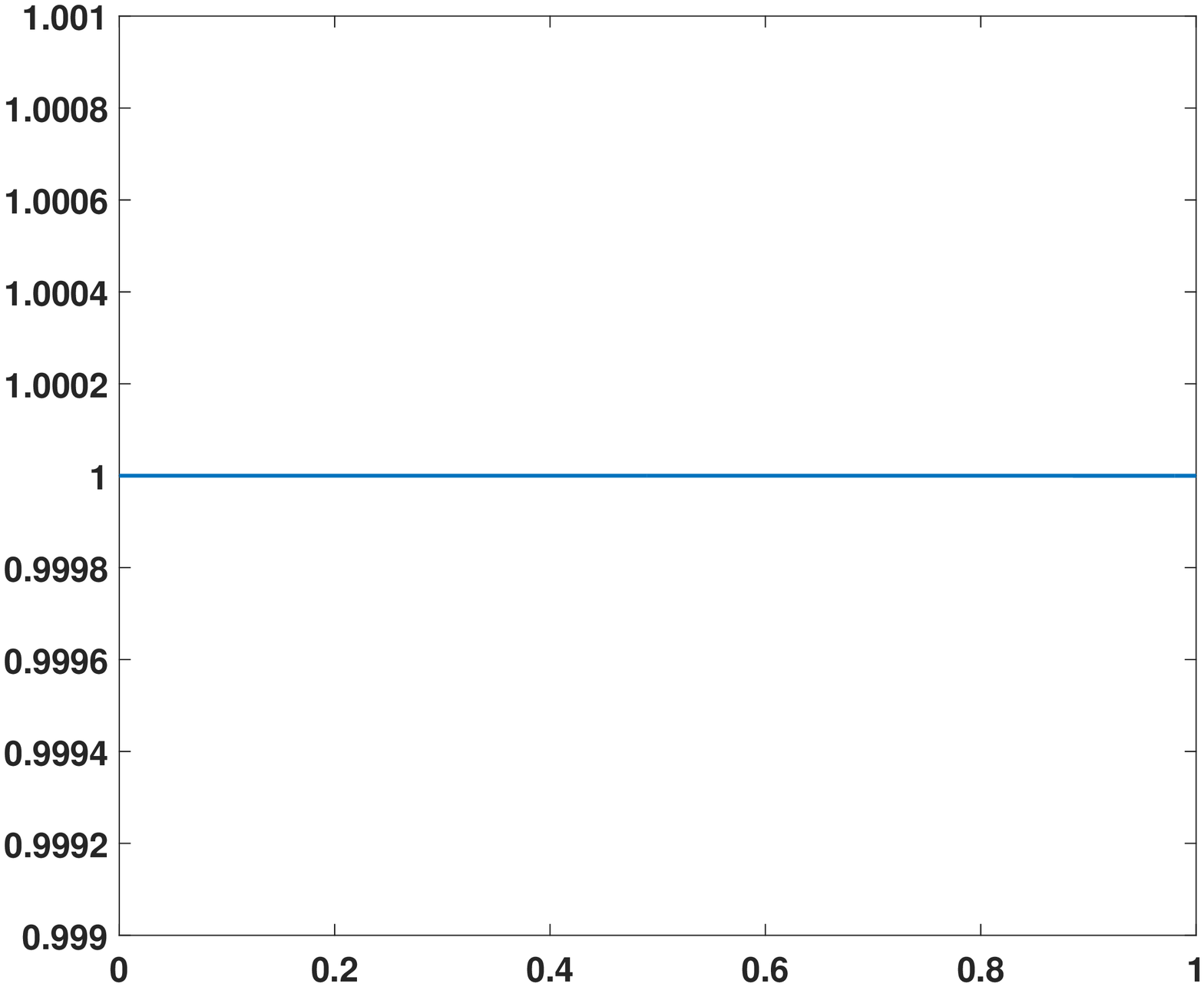}}
\subfigure[\scriptsize{$H(t)/H(0)$}]{
\includegraphics[width=0.4\textwidth]{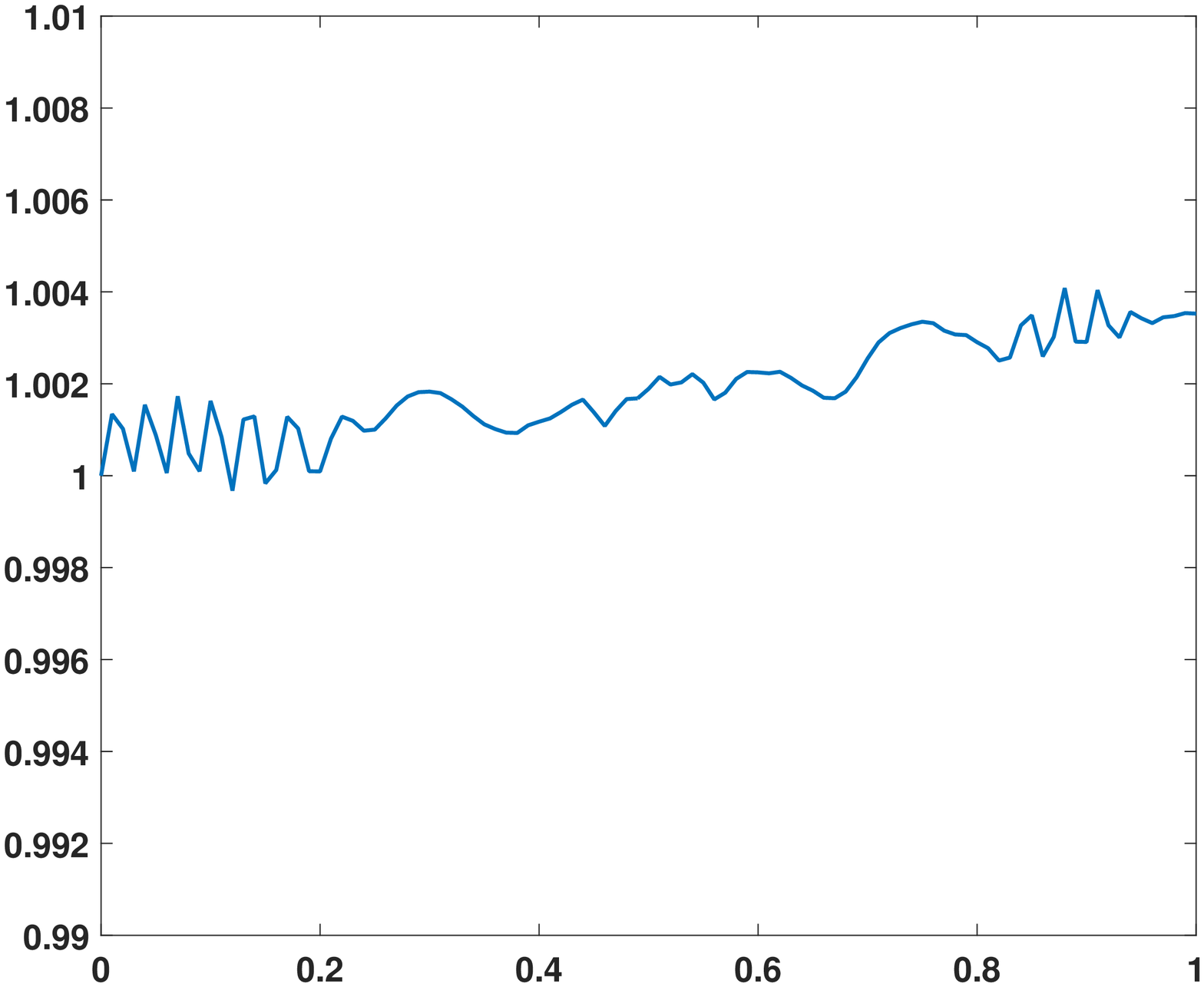}}
\caption{$V_{\Gamma}$ given by \eqref{KP}, $U$ given by \eqref{linear} and $\varepsilon=1/1024$}
\label{Fig. 7}
\end{figure}

\subsection{Numerical evidence for Anderson localization}
The phenomenon of Anderson localization describes the absence of dispersion for waves in random media with sufficiently strong random perturbations, which has been studied extensively.

To observe this phenomenon, we may consider a relatively weak periodic potential
\begin{equation}
V_{\Gamma}(x)=0.5+0.5\cos(x)
\end{equation}
and the random potential
\begin{equation}
U(x)=\sigma|z|\cos(x)
\end{equation}
where $\sigma$ represents the magnitude of randomness.

To measure the presence of Anderson localization, we consider the expectation of the second spatial moment of the position density, i.e.,
\begin{equation}
S(t)=\mathbb{E}[\int_{\mathbb{R}}|x|^2|\psi(t,x,z)|^2dx]
\end{equation}
It measures the spreading of the particle density. If the particles are localized, $S(t)$ should grow slowly and eventually become a constant in time.

Figure \ref{Fig. 8} shows the mean density of solutions with $\varepsilon=\frac{1}{4}$ at $T=1.5$ for different $\sigma$'s. And Figure \ref{Fig. 9} shows the temporal behavior of $S(t)$. From both figures, the phenomenon of Anderson localization is observed and our BD-SG can successfully capture this phenomenon.

\begin{figure}[H]
\centering
\subfigure[\scriptsize{$\sigma=0$}]{
\includegraphics[width=0.3\textwidth]{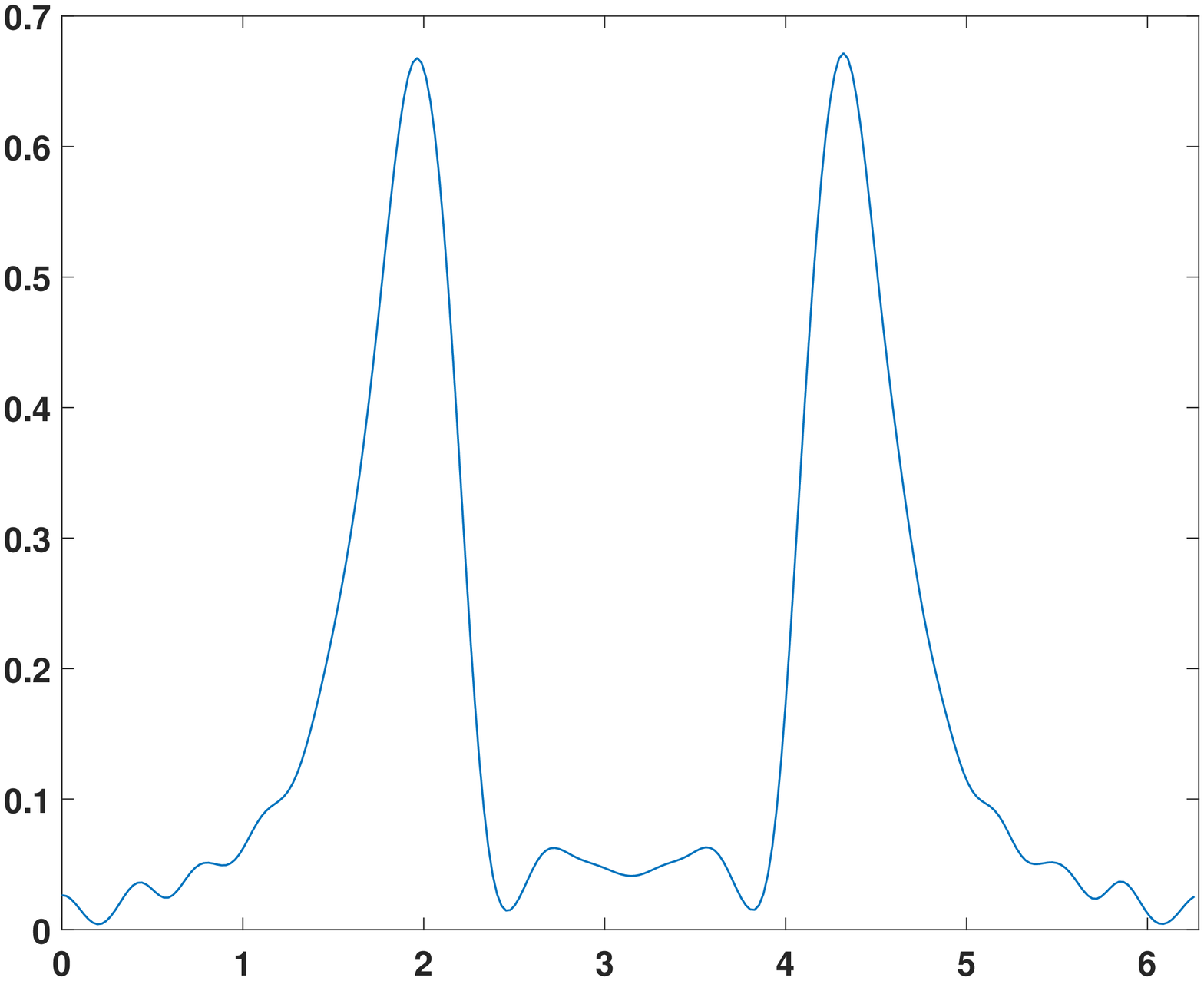}}
\subfigure[\scriptsize{$\sigma=3$}]{
\includegraphics[width=0.3\textwidth]{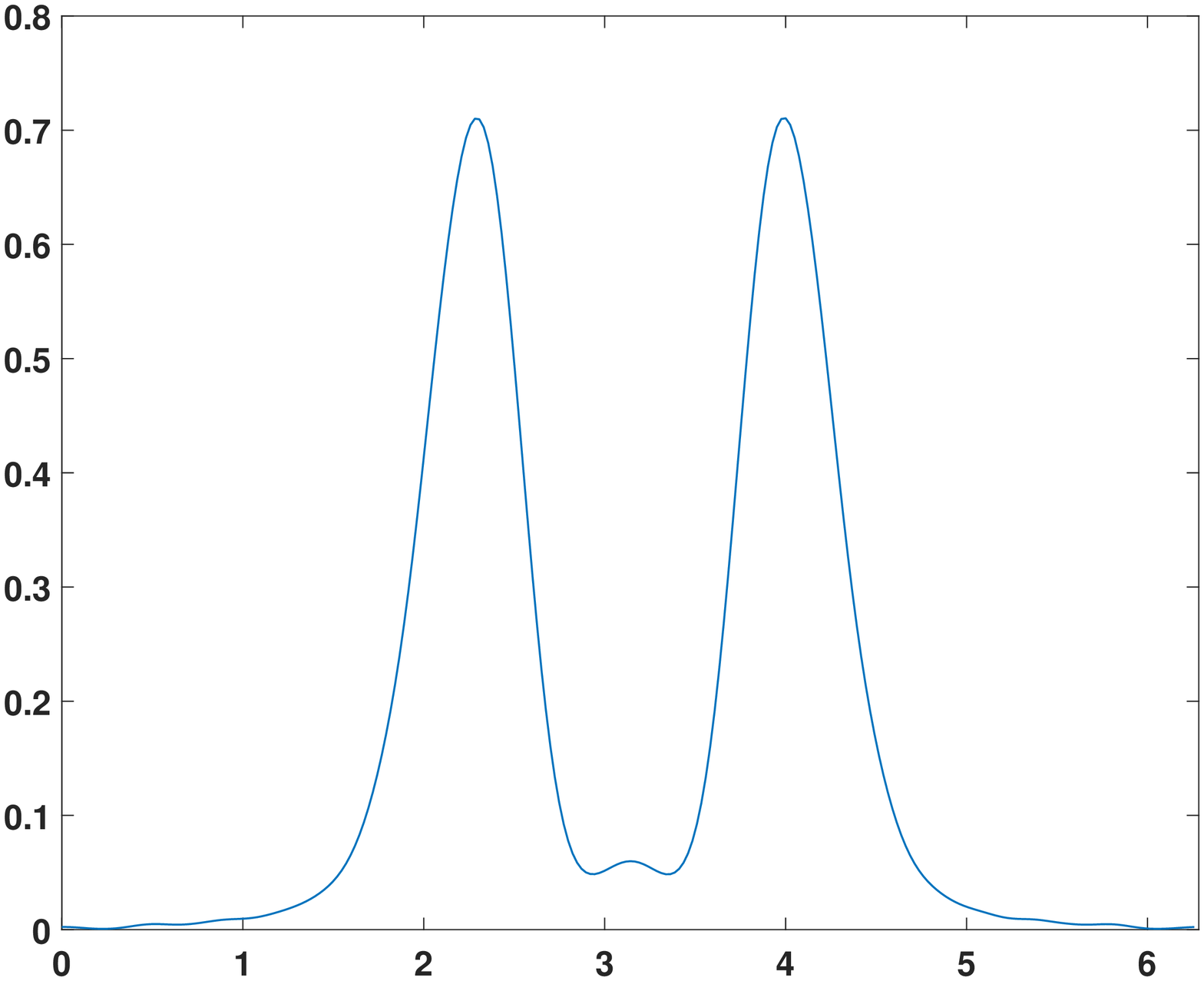}}
\subfigure[\scriptsize{$\sigma=5$}]{
\includegraphics[width=0.3\textwidth]{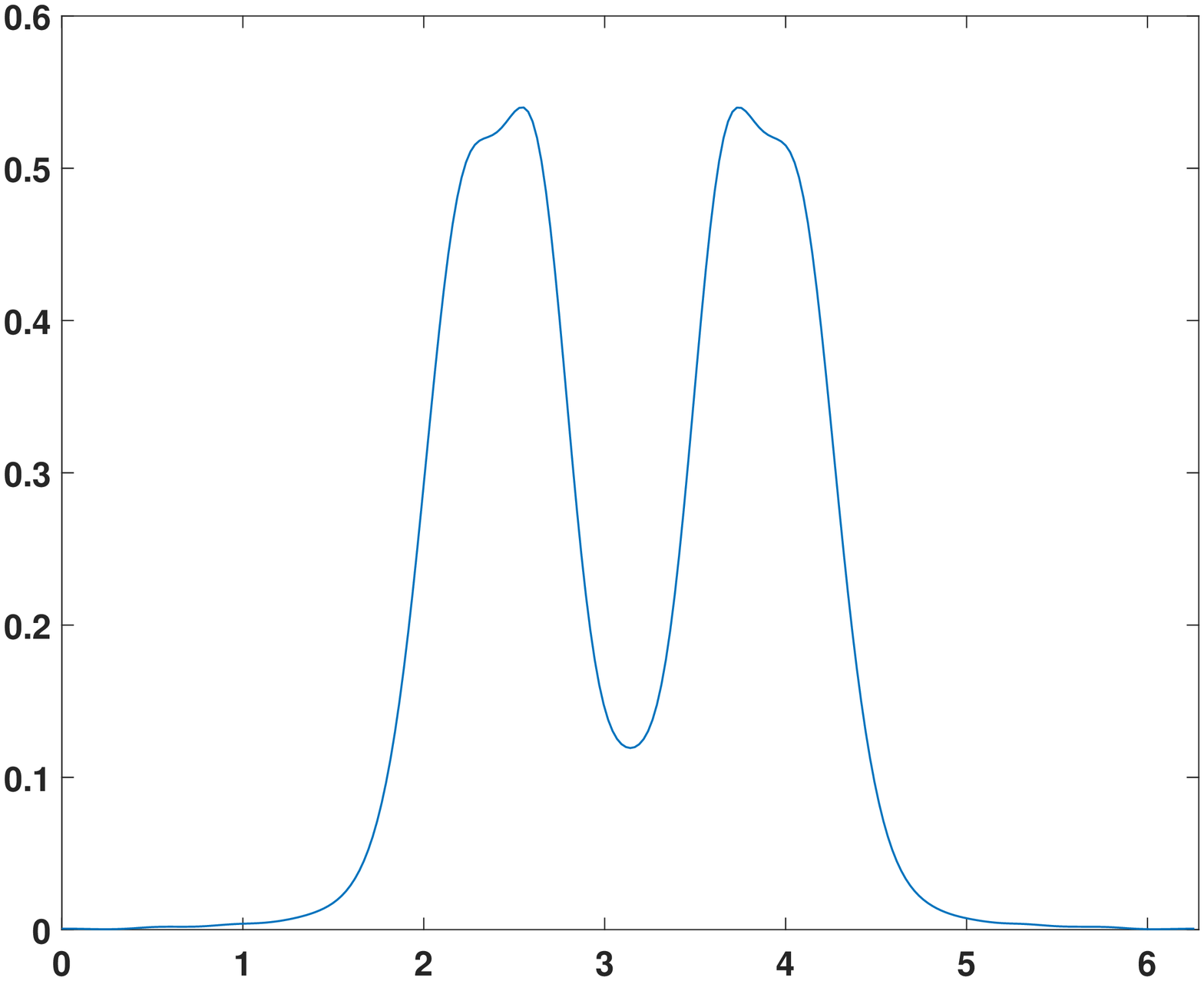}}
\caption{The expected particle density with $\varepsilon=1/4$}
\label{Fig. 8}
\end{figure}

\begin{figure}[H]
\centering
\includegraphics[width=0.5\textwidth]{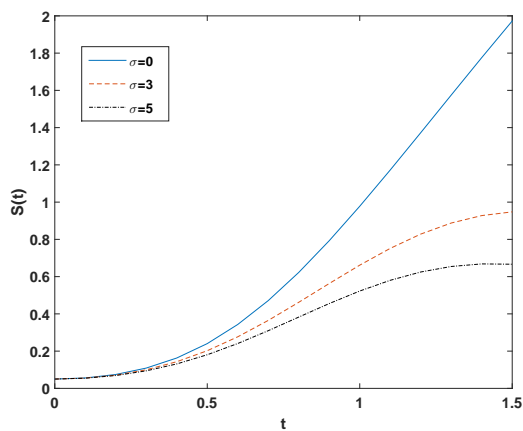}
\caption{The temporal behavior of $S(t)$}
\label{Fig. 9}
\end{figure}

\section{Conclusion}
\label{sec 6}
Here we present a new Bloch decomposition-based stochastic Galerkin algorithm for the one-dimensional Schr\"{o}dinger equation with a periodic lattice potential and a random external potential. We mainly focus on the case where the randomness is relatively weak. We first diagonalize the periodic part of the Hamiltonian operator using Bloch decomposition so that the effects of dispersion and the lattice potential are computed together. For the random potential, we utilize the general polynomial chaos expansion with a Galerkin procedure to form an ODE system which can be exactly solved. We prove analytically that this algorithm has the property of weak conservation of mass and thus it is unconditionally stable. Numerical experiments show that it is second order in time, achieves fast convergence in spatial discretization and gPC order and even has the property of weak conservation of energy. Moreover, we show the superiority of our method over the traditional time-splitting spectral method with Monte Carlo method and stochastic collocation method in the sense that the numerical cost of our method is much lower for the problem we consider, especially in the case where the periodic potential is non-smooth and/or $\varepsilon\ll 1$ since our method allows a larger time step size. Finally, we see that our method can successfully capture the phenomenon of Anderson localization.

For future extensions, we may study the approximation of orthogonal polynomials in a more theoretical way and expect some analytical results on the convergence rate. Furthermore, we may enlarge the space of random bases so that a more accurate basis for the square integrable functionals of random variables can be chosen.




\begin{thebibliography}{99}


\section*{References}

\bibitem{Anderson1958}
P.W. Anderson, Absence of diffusion in certain random lattices, Phys. Rev., 109(1958): 1492-1505.
\bibitem{Asch1998}
J. Asch, A. Knauf, Motion in periodic potentials, Nonlinearity, 11(1998): 175-200.
\bibitem{Ashcroft1976}
N.W. Ashcroft, N.D. Mermin, Solid State Physics, Saunders, New York, 1976.
\bibitem{Avishai1992}
Y. Avishai, R.M. Redheffer, Y. B. Band, Electron states in a magnetic field and random impurity potential: use of the theory of entire functions, J. Phys. A: Math. Gen., 25(1992): 3883-3889.
\bibitem{Back2011}
J. B\"{a}ck, F. Nobile, L. Tamellini, et al, Stochastic spectral Galerkin and collocation methods for PDEs with random coefficients: a numerical comparison, in Spectral and High Order Methods for Partial Differential Equations, Springer Berlin Heidelberg, 2011: pp. 43-62.
\bibitem{Bal1999}
G. Bal, A. Fannjiang, G. Papanicolaou, L. Ryzhik, Radiative transport in a periodic structure, J. Stat. Phys., 95(1999): 479-494.
\bibitem{Bal2006}
G. Bal, O. Pinaud, Accuracy of transport models for waves in random media, Wave Motion, 43(2006): 561-578.
\bibitem{Bal2011}
G. Bal, T. Komorowski, L. Ryzhik, Asymptotics of the solutions of the random Schr\"{o}dinger equation, Arch. Ration. Mech. An., 200(2011): 613-664.
\bibitem{Bao2002}
W.Z. Bao, S. Jin, P.A. Markowich, On time-splitting spectral approximations for the Schr\"{o}dinger equation in the semiclassical regime, J. Compu. Phys., 175(2002): 487-524.
\bibitem{Bao2003}
W.Z. Bao, S. Jin, P.A. Markowich, Numerical study of time-splitting spectral discretizations of nonlinear Schr\"{o}dinger equations in the semiclassical regimes, SIAM J. Sci. Comput., 25(2003): 27-64.
\bibitem{Bloch1928}
F. Bloch, \"{U}ber die quantenmechanik der elektronen in kristallgittern, Z. Phys., 52(1929): 555-600.
\bibitem{Blount1962}
E.I. Blount, Formalisms of band theory, in Solid State Physics, Vol. 13, Academic Press, New York, 1962: pp. 305-373.
\bibitem{Cameron1947}
R.H. Cameron, W.T. Martin, The orthogonal development of non-linear functionals in series of Fourier-Hermite functionals, Ann. Math., 1947: 385-392.
\bibitem{Fisher1991Thermal}
D.S. Fisher, D.A. Huse, Thermal fluctuations, quenched disorder, phase transitions, and transport in type-II superconductors, Phys. Rev. B, 43(1991): 130-159.
\bibitem{Fisher1991}
D.S. Fisher, D.A. Huse, Directed paths in a random potential, Phys. Rev. B, 43(1991): 10728-10742.
\bibitem{Fouqe2007}
J.P. Fouque, J.Garnier, G.Papanicolaou, K. S{\o}lna, Wave propagation and time reversal in randomly layered media,  Springer Science and Business Media, 2007.
\bibitem{Frohlich1983}
J. Fr\"ohlich, T. Spencer, Absence of diffusion in the Anderson tight binding model for large disorder or low energy, Commun. Math. Phys., 88(1983): 151-184.
\bibitem{Gosse II}
L. Gosse, Multiphase semiclassical approximation of an electron in a one-dimensional crystalline lattice II. Impurities, confinement and Bloch oscillations, J. Comput. Phys., 201(2004): 344-375.
\bibitem{Gosse I}
L. Gosse, P.A. Markowich, Multiphase semiclassical approximation of an electron in a one-dimensional crystalline lattice: I. homogeneous problems, J. Comput. Phys., 197(2004): 387-417.
\bibitem{Gosse III}
L. Gosse, N.J. Mauser, Multiphase semiclassical approximation of an electron in a one-dimensional crystalline lattice-III. From ab initio models to WKB for Schr\"{o}dinger-Poisson, J. Comput. Phys., 211(2006):  326-346.
\bibitem{Halperin1965}
B.I. Halperin, Green's functions for a particle in a one-dimensional random potential, Phys. Rev., 139(1965): A104-A117.
\bibitem{Hoverman2001}
F. Hovermann, H. Spohn, S. Teufel, Semiclassical Limit for the Schr\"{o}dinger Equation with a Short Scale Periodic Potential, Commun. Math. Phys., 215(2001): 609-629.
\bibitem{Hu2015}
J. Hu, S. Jin, D. Xiu. A stochastic Galerkin method for Hamilton¨CJacobi equations with uncertainty, SIAM J. Sci. Comput., 37(2015): A2246-A2269.
\bibitem{Huang2007}
Z. Huang, S. Jin, P.A. Markowich, C. Sparber, A Bloch decomposition-based split-step pseudospectral method for quantum dynamics with periodic potentials, SIAM J. Sci. Comput., 29(2007): 515-538.
\bibitem{Huang2008}
Z. Huang, S. Jin, P.A. Markowich, C. Sparber, Numerical simulation of the nonlinear Schr\"{o}dinger equation with multidimensional periodic potentials, Multiscale Model. Sim., 7(2008): 539-564.
\bibitem{Huang2010}
Z. Huang, S. Jin, P.A. Markowich, C. Sparber,  Bloch Decomposition method for waves in periodic media, in: Series in Contemporary Applied Mathematics, CAM 15, Some Problems on Nonlinear Hyperbolic Equations and Applications, Higher Education Press, Beijing, 2010: pp. 161-188.

\bibitem{Jin2011}
S. Jin, P.A. Markowich, C. Sparber, Mathematical and computational methods for semiclassical Schr\"{o}dinger equations, Acta Numer., 20(2011): 121-209.
\bibitem{Jin2015}
S. Jin, D. Xiu, X. Zhu, Asymptotic-preserving methods for hyperbolic and transport equations with random inputs and diffusive scalings, J. Comput. Phys., 289(2015): 35-52.
\bibitem{Johnson1986}
D.D. Johnson, D.M. Nicholson, F.J. Pinski, et al, Density-functional theory for random alloys: Total energy within the coherent-potential approximation, Phys. Rev. Lett., 56(1986): 2088-2091.
\bibitem{Kirsh1982}
W. Kirsch, F. Martinelli, On the density of states of Schr\"{o}dinger operators with a random potential, J. Phys. A: Math. Gen., 15(1982): 2139-2156.
\bibitem{Kirsh1982Spectrum}
W. Kirsch, F. Martinelli, On the spectrum of Schr\"{o}dinger operators with a random potential, Commun. Math. Phys., 85(1982): 329-350.
\bibitem{Kirsh1989}
W. Kirsch, Random Schr\"{o}dinger operators a course, in Schr\"{o}dinger operators, Springer Berlin Heidelberg, 1989: pp. 264-370.
\bibitem{Leschke2005}
H. Leschke, P. Muller, S. Warzel, A survey of rigorous results on random Schr\"{o}dinger operators for amorphous solids, in Interacting Stochastic Systems, Springer Berlin Heidelberg, 2005: pp. 119-151.
\bibitem{Luttinger1951}
J.M. Luttinger, The effect of a magnetic field on electrons in a periodic potential, Phys. Rev., 84(1951): 814-817.
\bibitem{Markowich1994}
P.A. Markowich, N.J. Mauser, F. Poupaud, A Wigner-function approach to \(semi\) classical limits: Electrons in a periodic potential, J. Math. Phys., 35(1994): 1066-1094.
\bibitem{Panati2003}
G. Panati, H. Spohn, S. Teufel, Effective dynamics for Bloch electrons: Peierls substitution and beyond, Commun. Math. Phys., 242(2003): 547-578.
\bibitem{Roati2008}
G. Roati, C. D¡¯Errico, L. Fallani, et al, Anderson localization of a non-interacting Bose¨CEinstein condensate, Nature, 453(2008): 895-898.
\bibitem{Teufel2003}
S. Teufel, Adiabatic perturbation theory in quantum dynamics, Springer Science and Business Media, 2003.
\bibitem{Wilcox1978}
C.H. Wilcox, Theory of Bloch waves, J. Ann. Math., 33(1978): 146-167.
\bibitem{Xiu2002}
D. Xiu, G.E. Karniadakis, The Wiener-Askey polynomial chaos for stochastic differential equations, SIAM J. Sci. Comput., 24(2002): 619-644.
\bibitem{Xiu2005}
D. Xiu, J.S. Hesthaven, High-order collocation methods for differential equations with random inputs, SIAM J. Sci. Comput., 27(2005): 1118-1139.
\bibitem{Xiu2009}
D, Xiu, Fast numerical methods for stochastic computations: a review, Commun. Comput. Phys., 5(2009): 242-272.
\bibitem{Xiu2010}
D. Xiu, Numerical methods for stochastic computations: a spectral method approach, Princeton University Press, 2010.
\bibitem{Zak1968}
J. Zak, Dynamics of electrons in solids in external fields, Phys. Rev., 168(1968): 686-695.

\end{thebibliography}


\end{document}